\documentclass[a4paper,11pt,pdf]{amsart}
\usepackage{enumerate, amsmath, amsfonts, amssymb, amsthm, thmtools, wasysym, graphics, graphicx, xcolor, frcursive,comment,bbm}

\usepackage{etex}
\usepackage{lscape}
\usepackage{tikz-cd}

\makeatletter
\newcommand{\thickhline}{%
    \noalign {\ifnum 0=`}\fi \hrule height 1pt
    \futurelet \reserved@a \@xhline
}

\definecolor{darkblue}{rgb}{0.0,0,0.7} 
 
\definecolor{darkred}{rgb}{0.7,0,0} 
\usepackage{hyperref}
\usepackage[all]{xy}
\usepackage[T1]{fontenc}
\usepackage{adjustbox}

\usepackage{vmargin}            
\setmarginsrb{3cm}{2.5cm}{3cm}{2.5cm}{0cm}{0.6cm}{0cm}{0cm}

\usepackage{caption,lipsum}
\usepackage{graphicx}                  
\usepackage{pstricks,pst-plot,pst-text,pst-tree,pst-eps,pst-fill,pst-node,pst-math}
\usepackage{setspace}
\usepackage{multicol}

\newcommand{\darkred}{\color{darkred}} 
\newcommand{\defn}[1]{\emph{\darkred #1}}

\def\H{\mathcal{H}}

\usepackage{etex}

\newtheorem{theorem}{Theorem}[section]
\newtheorem{prop}[theorem]{Proposition}
\newtheorem{lemma}[theorem]{Lemma}
\newtheorem{cor}[theorem]{Corollary}

\theoremstyle{definition}
\newtheorem{definition}[theorem]{Definition}
\newtheorem{rmq}[theorem]{Remark}
\newtheorem{exple}[theorem]{Example}
\newtheorem{question}[theorem]{Question}

\numberwithin{equation}{section}

\title[Toric reflection groups]{Toric reflection groups}

\author{Thomas Gobet}
\address{Institut Denis Poisson, CNRS UMR 7350, Faculté des Sciences et Techniques, Université de Tours, Parc de Grandmont, 
37200 TOURS, France}
\email{thomas.gobet@lmpt.univ-tours.fr}

\begin{document}
\maketitle

\begin{abstract}
Several finite complex reflection groups have a braid group which is isomorphic to a torus knot group. The reflection group is obtained from the torus knot group by declaring meridians to have order $k$ for some $k\geq 2$, and meridians are mapped to reflections. We study all possible quotients of torus knot groups obtained by requiring meridians to have finite order. Using the theory of $J$-groups of Achar and Aubert, we show that these groups behave like (in general infinite) complex reflection groups of rank two. The large family of "toric reflection groups" which we obtain includes, among others, all finite complex reflection groups of rank two with a single conjugacy class of reflecting hyperplanes, as well as Coxeter's truncations of the $3$-strand braid group. We classify these toric reflection groups and explain why the corresponding torus knot group can be naturally considered as its braid group. In particular, this yields a new infinite family of reflection-like groups admitting a braid group which is a Garside group. Moreover, we show that a toric reflection group has cyclic center by showing that the quotient by the center is isomorphic to the alternating subgroup of a Coxeter group of rank three. To this end we use the fact that the center of the alternating subgroup of an irreducible, infinite Coxeter group of rank at least three is trivial. Several ingredients of the proofs are purely Coxeter-theoretic, and might be of independent interest.   
\end{abstract}

\tableofcontents

\thispagestyle{empty}

\section{Introduction}

The $3$-strand braid group has many interesting quotients admitting realizations as reflection groups of rank two over $\mathbb{R}$ or $\mathbb{C}$, the most famous one being the symmetric group $\mathfrak{S}_3$. The symmetric group $\mathfrak{S}_3$ is a finite real reflection group or finite Coxeter group, and the $3$-strand braid group is isomorphic to the corresponding Artin--Tits group. In this case the quotient map from $\mathcal{B}_3$ to $\mathfrak{S}_3$ maps Dehn twists to reflections. This situation admits a well-known generalization, where the symmetric group is replaced by any (not necessarily finite) Coxeter group and the braid group by the corresponding Artin--Tits group (see for instance~\cite[Section 6.6]{KT} for an introduction to the topic). In this situation, one has a natural recipe to attach to a Coxeter group a group playing the role of the braid group, and sharing the same kind of (in general conjectural) properties---absence of torsion, solvable word and conjugacy problem, etc.

In the complex case, that is, if we start with a finite \textit{complex} reflection group, there is also a way to attach a group with the same kind of properties as the braid group: one defines the braid group attached to a finite complex reflection group as the fundamental group of the space of regular orbits in the natural representation of the reflection group~\cite{BMR}. This allows one to recover the $3$-strand braid group from several finite complex reflection groups of rank two arising as quotients---like for instance $G_4$ in Shephard and Todd notation. In the cases where the reflection group can be realized over the real numbers, \emph{i.e.}, in the cases where it is a finite Coxeter group, the obtained braid group is isomorphic to the Artin--Tits group of the Coxeter group~\cite{Brieskorn}. In the complex case, the generalizations of Dehn twists are given by so-called braided reflections, which are also mapped to reflections in the quotients. 

In this paper, we study another kind of generalizations of reflection quotients of the $3$-strand braid group. A feature of the $3$-strand braid group is that it is isomorphic to the knot group of the trefoil knot~\cite[Section 1.1.4]{KT}, which is a torus knot (\emph{i.e.}, a knot which lies on the surface of a torus). Several finite complex reflection groups with a single conjugacy class of reflecting hyperplanes have their braid group isomorphic to a torus knot group~\cite{Bannai}, and the quotient map from the torus knot group to the complex reflection group maps meridians---which one can consider as generalizations of Dehn twists---to reflections. Moreover, as in the aforementioned cases, a presentation of the reflection group is obtained from a presentation of the torus knot group having meridians as generators by setting $x^k=1$ for some $k\geq 2$ for any (equivalently every) meridian $x$. The obtained reflection groups---which include the aforementioned quotients of the $3$-strand braid group---are of rank two, \emph{i.e.}, they are reflection groups over $\mathbb{C}^2$. The aim of this paper is to study all possible quotients of all possible torus knot groups arising in this way. Note that surjective maps from link groups onto Coxeter or reflection groups, which have the property to send meridians to reflections, arose in a knot-theoretic setting in the study of the so-called \textit{meridional rank conjecture} stating an equality between the bridge number of a link $L$ and its meridional rank (see~\cite{Baader_1, Baader_2}). Indeed, if one finds such a surjective map from $\Gamma=\pi_1(S^3\backslash L)$ onto $W$, where $L$ is a link and $W$ a reflection group, then the meridian rank of $L$, that is, the minimal number of meridians needed to generate $\Gamma$, is necessarily bounded below by the reflection rank of $W$, that is, the minimal number of reflections needed to generate $W$. This conjecture is solved for torus links~\cite{merid_torus}.  

The aforementioned quotients of torus knot groups studied in this paper are infinite in general. We show that they are generalizations of complex reflection groups of rank two, with the torus knot group playing the role of an attached "braid group", and give a few fundamental results on their structure as well as a classification.

\medskip

To be more precise, let $n,m\geq 2$ be two relatively prime integers, with $n<m$. The \defn{torus knot group} $G(n,m)$ is the fundamental group of the complement of the torus knot $T_{n,m}$ in $\mathbb{R}^3$. It admits (see for instance~\cite[Chapter 3, Section C]{Rolfsen}) the well-known presentation \begin{equation}\label{pres_0} G(n,m)=\langle \ x, y \ \vert\ x^n=y^m \ \rangle.\end{equation} Another presentation, which we shall call \textit{classical}, is given by \begin{equation}\label{pres_2}\langle \ x_1, x_2, \dots, x_n \ \vert\ 
\underbrace{x_1 x_2 \cdots}_{m~\text{factors}} = \underbrace{x_2 x_3\cdots}_{m~\text{factors}} = \dots = \underbrace{x_n x_1 \cdots}_{m~\text{factors}}
\ \rangle,\end{equation} where indices are taken modulo $n$. See for instance~\cite[6.4, Problem 10]{Dual}\footnote{An explicit isomorphism between $G(n,m)$ and the group with presentation~\ref{pres_2} is given by $x\mapsto x_1 x_2 \cdots x_m$, $y\mapsto x_1 x_2 \cdots x_n$. Note that in Presentation~\ref{pres_2}, the generators are meridians. To see this, one can for instance use~\cite[Proposition 3.38(b)]{Knots}, where it is shown that if $a,b\in\mathbb{Z}$ are such that $an-bm=1$, then $y^a x^{-b}$ is a meridian. Using the defining relations of Presentation~\ref{pres_2}, we see that under the isomorphism given above, the element $y^a x^{-b}$ is mapped to $x_i^{\pm 1}$ for some $i$, hence there is an $i$ such that $x_i$ is a meridian. But since $n$ and $m$ are coprime, all $x_i$'s are conjugate, hence they are all meridians.}. Using the fact that the torus knots $T_{n,m}$ and $T_{m,n}$ are isotopic, one obtains a third presentation, where generators are also meridians, which we shall call \textit{dual}, given by \begin{equation}\label{pres_3}\langle \ y_1, y_2, \dots, y_m \ \vert\ 
\underbrace{y_1 y_2 \cdots}_{n~\text{factors}} = \underbrace{y_2 y_3\cdots}_{n~\text{factors}} = \dots = \underbrace{y_m y_1 \cdots}_{n~\text{factors}}\ \rangle. \end{equation}    
In the case where $n=2$, $m=3$ (more generally for $m$ odd), the classical presentation above is the classical presentation of the $3$-strand braid group $\mathcal{B}_3$ (more generally, of the Artin--Tits group of dihedral type $I_2(m)$), while the dual presentation is Birman-Ko-Lee's presentation~\cite{BKL} (more generally the \textit{dual} presentation in the sense of~\cite{Dual}). This explains our terminology. Presentation~\ref{pres_0} is obtained from the classical one by setting $x=x_1 x_2 x_1$, $y=x_1 x_2$, while the third one is obtained by setting $y_1=x_1$, $y_2=x_2$, $y_3=x_1 x_2 x_1^{-1}$. The symmetric group $\mathfrak{S}_3$ is obtained from either the classical or dual presentations by adding the relations $a^2=1$ for all generators $a$ (or for a single generator $a$, as they are all conjugate). The complex reflection groups mentioned above are obtained in a similar way by setting $a^k=1$ for the generators $a$ and some $k\geq 2$, and the images of the generators (and more generally the meridians) in the quotient are reflections. More precisely, in Shephard-Todd notation 
\begin{itemize}
\item $G_4$, $G_8$, respectively $G_{16}$ are obtained from the classical presentation of $\mathcal{B}_3\cong G(2,3)$ by adding the relations $x_i^k=1$ for all $i$, where $k=3, 4$ and $5$ respectively,
\item $G_{12}$ is obtained from the classical presentation of $G(3,4)$ by adding the relations $x_i^2=1$ for all $i$,
\item $G_{22}$ is obtained from the classical presentation of $G(3,5)$ by adding the relations $x_i^2=1$ for all $i$,
\item $G_{20}$ is obtained from the classical presentation of $G(2,5)$ by adding the relation $x_i^3=1$ for all $i$,
\item For odd $m$, the group $G(m,m,2)$, which is also the dihedral Coxeter group of type $I_2(m)$, is obtained from $G(2,m)$ by adding the relation $x_i^2=1$ for all $i$. Note that in this case $G(2,m)$ is the Artin--Tits group of type $I_2(m)$. 
\end{itemize} 
Generalizing the groups obtained in the first point above, Coxeter~\cite{coxeter_factor} studied the quotient of $\mathcal{B}_3$ (and more generally $\mathcal{B}_n$) by the relations $x_1^k=x_2^k=1$ where $k\geq 2$, and showed that this quotient is finite if and only if $k\leq 5$. He also showed that these groups, which are sometimes called \textit{truncated braid groups}, admit a complex representation as groups generated by (pseudo-)reflections. For $k\leq 5$ he showed that this representation is faithful, and that the group is finite if and only if $k\leq 5$. 

Let $n,m$ be as above and $k\geq 2$. We define a three-parameter family of groups generalizing all the examples given above, called \defn{toric reflection groups}, by setting \begin{equation}\label{pres_wknm}W(k,n,m):=\bigg\langle \ x_1, x_2, \dots, x_n \ \bigg\vert\ 
\begin{matrix} x_i^k=1~\text{for}~i=1, \dots, n,\\ \underbrace{x_1 x_2 \cdots}_{m~\text{factors}} = \underbrace{x_2 x_3\cdots}_{m~\text{factors}} = \dots = \underbrace{x_n x_1 \cdots}_{m~\text{factors}} \end{matrix} 
\ \bigg\rangle.\end{equation}

We define the conjugates of the non-trivial powers $x_i^\ell$ (\emph{i.e.}, not equal to the identity) of the generators $x_i$ to be the \defn{reflections} in $W(k,n,m)$. This is partly justified by the fact that in the aforementioned cases where $W(k,n,m)$ is finite, the $x_i$'s are reflections, and more generally by the following fact. We show in Theorem~\ref{thm_11} below that these groups are isomorphic to groups which are part of a family of groups introduced by Achar and Aubert, called \textit{$J$-groups}~\cite{AA}, which are generalizations of complex reflection groups of rank two---see Section~\ref{jgroup} below for precise definitions, and Theorem~\ref{part_I_main} together with Corollary~\ref{coro_single_class} for a more precise and complete statement. Each such group admits a representation as a subgroup of $\mathrm{GL}_2(\mathbb{C})$ generated by (pseudo-)reflections~\cite[Section 4]{AA}. Achar and Aubert showed that a $J$-group is finite if and only if it is a finite complex reflection group of rank two~\cite[Theorem 1.2]{AA}.  

\begin{theorem}[Toric reflection groups are $J$-groups]\label{thm_11}
The group $W(k,n,m)$ is isomorphic to the $J$-group $J\begin{pmatrix} k & n & m \\  & n & m \end{pmatrix}$ of Achar and Aubert. Under this isomorphism, the generators of $W(k,n,m)$ correspond to elements of the $J$-group acting by reflections in Achar and Aubert's representation.  
\end{theorem}

This allows one to consider Presentation~\ref{pres_wknm} as a presentation of a reflection group in some sense. Nevertheless, let us point out that Achar and Aubert's representation is \textit{not} faithful in general: see Section~\ref{rep} below. 

We will say that two toric reflection groups $W(k,n,m)$, $W(k',n',m')$ with respective sets of reflections $R, R'$ are \defn{reflection isomorphic}, written $W(k,n,m)\cong_{\mathrm{ref}}W(k',n',m')$, if there is a group isomorphism $\varphi: W(k,n,m)\longrightarrow W(k',n',m')$ such that $\varphi(R)=R'$. The following statement classifies toric reflection groups (see Theorem~\ref{thm_clas} below)    

\begin{theorem}[Classification of toric reflection groups]\label{thm:classification}
Let $k,k',n,n',m,m'\geq 2$ with $n<m$, $n'<m'$, $n$ and $m$ coprime, and $n'$ and $m'$ coprime. Then $$W(k,n,m)\cong_{\mathrm{ref}} W(k',n',m')\Leftrightarrow k=k',n=n',\text{~and~}m=m'.$$
\end{theorem}

Note that Presentation~\ref{pres_wknm} can be given as well if $n>m$. But in this case, thanks to the isomorphism $G(n,m)\cong G(m,n)$ and the fact that the toric reflection group is obtained from the torus knot group by killing the $k^{th}$ power of meridians, we get that $W(k,n,m)\cong W(k,m,n)$. Hence the case $n<m$ is sufficient to parametrize all toric reflection groups, and with this assumption, Theorem~\ref{thm:classification} says that for any toric reflection group $W$, there is a single pair $(k,n,m)$ with $W\cong_{\mathrm{ref}} W(k,n,m)$.  

As an immediate corollary of Theorem~\ref{thm:classification}, we get the following definition of the braid group of a toric reflection group, in the spirit of Achar and Aubert's characterization of finite $J$-groups: 
 
\begin{cor}[Braid groups of toric reflection groups]\label{cor_def_bd}
Let $W$ be a toric reflection group with set of reflections $R$. Let $k,n,m\geq 2$ with $n<m$ and $n,m$ coprime such that $W\cong_{\mathrm{ref}} W(k,n,m)$. Set $\mathcal{B}(W,R):=G(n,m)$. Then 
\begin{enumerate}
\item The group $\mathcal{B}(W,R)$ is well-defined, \emph{i.e.}, only depends on the isomorphism class of toric reflection group of $W$. 
\item If $W$ is finite, then $\mathcal{B}(W,R)$ is isomorphic to the braid group of the complex reflection group $W$. 
\end{enumerate}
\end{cor}

Note that Schreier~\cite{Schreier} prove that two torus knot groups $G(n,m)$ and $G(n',m')$ ($n<m$ and $n'<m'$) are isomorphic (as abstract groups) if and only if $n=n'$ and $m=m'$ (his result is actually more general as he does not assume the parameters to be coprime). 

The above corollary yields the generalization of reflection quotients of the $3$-strand braid group announced at the beginning of the introduction. Note that every torus knot group is a \textit{Garside group}~\cite[Example 4]{DP} (see also~\cite[Section 3]{Gobet}, and~\cite{Garside} for basics on Garside groups), hence it shares many properties with Artin's braid group and more generally Artin--Tits groups of spherical type, \emph{i.e.}, attached to finite Coxeter groups, which are all Garside groups. Such groups have for instance solvable word and conjugacy problems, and are torsion-free. Note that Artin--Tits groups attached to infinite Coxeter groups are not Garside groups in general (some of them are known to be so-called \textit{quasi-Garside} groups~\cite{Garside}), while by Corollary~\ref{cor_def_bd}, an infinite toric reflection group has an attached "braid group" which is always a Garside group. This is not especially surprising and could already be observed for Coxeter's truncated braid groups~\cite{coxeter_factor}, as toric reflection groups have, even when infinite, properties similar to finite (irreducible) Coxeter groups. For instance, they have a non-trivial cyclic center---see~Corollary~\ref{cor_cent} below.  

A main ingredient in the proof of Theorem~\ref{thm:classification} is the theory of Coxeter groups and their parabolic subgroups, especially in rank three. In the sequel, we show several results which are purely Coxeter-theoretic, and might be of independent interest. To establish Theorem~\ref{thm:classification}, we identify the quotient of a toric reflection group $W$ by its center. More precisely, let $k,n,m$ be as above, and let $W_{k,n,m}$ be the Coxeter group of rank three whose diagram is a triangle with edges labelled by $k,n$, and $m$. Let $W_{k,n,m}^+$ be its alternating subgroup, \emph{i.e.}, the subgroup of elements with signature $1$. Let $c=(x_1 x_2 \cdots x_n)^m\in W(k,n,m)$, which is central in $W(k,n,m)$. Then we show (see Theorem~\ref{thm_alt} below)

\begin{theorem}\label{thm_ses}
The group $W(k,n,m)$ is a central extension of $W_{k,n,m}^+$ by the subgroup $\langle c \rangle$. That is, we have a short exact sequence $$1\longrightarrow \langle c \rangle \longrightarrow W(k,n,m)\longrightarrow W_{k,n,m}^+ \longrightarrow 1$$ 
\end{theorem}

In most cases, the group $W_{k,n,m}$ is infinite and irreducible. The determination of the center of a toric reflection groups uses the following result, which is of independent interest (see Proposition~\ref{prop_2} below).

\begin{prop}[Center of alternating subgroups of Coxeter groups]\label{prop_centre_coxeter}
Let $(W,S)$ be a Coxeter system of rank at least $3$. Let $W^+$ be the alternating subgroup of $W$. Then the center $Z(W^+)$ of $W^+$ is included in the center of $W$. In particular, if $(W,S)$ is infinite, irreducible and of rank at least three, then $Z(W^+)$ is trivial. 
\end{prop}

Together with~\ref{thm_ses} and a case-by-case check in the cases where $W_{k,n,m}$ is finite, Proposition~\ref{prop_centre_coxeter} yields (see Theorem~\ref{thm_alt} below):

\begin{cor}[Center of toric reflection groups]\label{cor_cent}
The center of $W(k,n,m)$ is cyclic, generated by $c$. 
\end{cor}

In the case where $W(k,n,m)$ is infinite, we do not know if $c$ has finite order or not (see Remark~\ref{center_finite} below). An explicit identification of the center together with Theorem~\ref{thm_ses} would show that the groups $W(k,n,m)$ have a solvable word problem---see Question~\ref{quest_solv} below and the discussion above it. 

In the case where $W(k,n,m)$ is finite, Theorem~\ref{thm_ses} together with Corollary~\ref{cor_cent} give for the groups listed above a new and more general explanation for the known description of the quotient $W(k,n,m)/Z(W(k,n,m))$, as we see that it is isomorphic to the alternating subgroup of a Coxeter group which can be attached in a uniform way to all the concerned finite complex reflection groups. Note that, when $W(k,n,m)$ is finite, it is known that $k-1$ is the number of conjugacy classes of reflections in $W(k,n,m)$, and that $n$ is the \textit{reflection rank} of $W(k,n,m)$, that is, the minimal number of reflections which are needed to generate $W(k,n,m)$. For arbitrary $W(k,n,m)$, we shall see that it is still true that $k-1$ is the number of conjugacy classes of reflections, but we do not know if $n$ is equal to the reflection rank of $W(k,n,m)$ or not---see Remark~\ref{interpret_param}. A positive answer would be a first step towards another proof of the classification of toric reflection groups given in Theorem~\ref{thm:classification} which may avoid the recourse to Coxeter groups---which has other advantages, as for instance Coxeter groups of rank three have nice geometric realizations. This would also give another way of showing that the meridional rank of the torus knot $T_{n,m}$ ($n<m$) is equal to $n$ (established in~\cite{merid_torus}).

\medskip

The paper is organized as follows. Note that, since the results are not always proven in the same order as they appear in the introduction, and sometimes require more notation than what we introduced above, we have a different numbering of the results in the rest of the paper, and sometimes a slightly different formulation. In Section~\ref{toricj}, we recall from~\cite{AA} a few basic results on $J$-groups, and show that toric reflection groups are $J$-groups, allowing one to consider toric reflection groups as (a generalization of) complex reflection groups of rank two. In Section~\ref{center}, we identify the center of toric reflection groups and link them as explained above to alternating subgroups of Coxeter groups of rank three. In Section~\ref{main}, we use the previously established results to classify toric reflection groups; a main ingredient to this end is the study of finite subgroups of (alternating subgroups of) Coxeter groups of rank three.

\section{$J$-groups and toric reflection groups}\label{toricj}

In this section, we first recall the definition and basic properties of $J$-groups in Subsection~\ref{jgroup} below. These groups, which are infinite in general, were introduced by Achar and Aubert in~\cite{AA} as a generalization of finite complex reflection groups of rank two. They are defined as certain normal subgroups of a family of groups defined by generators and relations, which are themselves $J$-groups. We then show that toric reflection groups are $J$-groups of a certain kind in Subsection~\ref{tarej} by giving an explicit presentation by generators and relations for these $J$-groups using the Reidemeister-Schreier algorithm. Finally, in Subsection~\ref{rep}, we discuss the faithfulness of Achar and Aubert's representation for this family of groups. 

\subsection{$J$-groups: definition and basic properties}\label{jgroup}

Let $a,b,c\geq 2$. Let $J\begin{pmatrix} a & b & c \\ ~ & ~ & ~ \end{pmatrix}$ be the group defined by the following presentation \begin{align*}J\begin{pmatrix} a & b & c \\ ~ & ~ & ~ \end{pmatrix}:=\langle \ s,t,u \ \vert \ s^a = t^b = u^c=1,\ stu=tus=ust \ \rangle.\end{align*}
Let $a',b',c'$ be three pairwise coprime positive integers, such that $k'$ divides $k$ for all $k\in\{a,b,c\}$. Let $J\begin{pmatrix} a & b & c \\ a' & b' & c' \end{pmatrix}$ be the normal closure in $J\begin{pmatrix} a & b & c \\ ~ & ~ & ~ \end{pmatrix}$ of the elements $s^{a'}, t^{b'}$ and $u^{c'}$. These groups were defined by Achar and Aubert in~\cite{AA}, and are called \defn{$J$-groups}. Note that $J\begin{pmatrix} a & b & c \\ 1 & 1 & 1 \end{pmatrix}=J\begin{pmatrix} a & b & c \\~ & ~ & ~ \end{pmatrix}$, hence $J\begin{pmatrix} a & b & c \\~ & ~ & ~ \end{pmatrix}$ is itself a $J$-group, and we may call it the \defn{parent} $J$-group of $J\begin{pmatrix} a & b & c \\ a' & b' & c' \end{pmatrix}$; in general, and also for other $J$-groups, $1$'s will be omitted in the second row of parameters.

Recall that a complex reflection group is a (finite) subgroup $W\subseteq \mathrm{GL}_n(\mathbb{C})$ generated by (pseudo-)reflections, that is, elements of finite order whose space of fixed points is a hyperplane in $\mathbb{C}^n$---see~\cite{Broue, LT} for basics on complex reflection groups. Achar and Aubert's main result is the following

\begin{theorem}[{see~\cite[Theorem 1.2]{AA}}]
A $J$-group is finite if and only if it is a finite complex reflection group of rank two.
\end{theorem}

In the case where a $J$-group $H:=J\begin{pmatrix} a & b & c \\ a' & b' & c' \end{pmatrix}$ is finite, it turns out that both $H$ and $G:=J\begin{pmatrix} a & b & c \\ ~ & ~ & ~ \end{pmatrix}$ are finite (see~\cite{AA}). In this case, the generators $s, t$ and $u$ are reflections, and since a power of a reflection is either the identity or a reflection and the set of reflections is stable by conjugation, the subgroup $H$ is generated by reflections, \emph{i.e.}, it is a reflection subgroup of $G$. In the case where $G$ is infinite, it is shown in~\cite{AA} that one can still define a representation $$\rho : G\longrightarrow \mathrm{GL}_2(\mathbb{C})$$ of $G$ (and hence of $H$) where the generators of $G$ act by reflections. But as we shall see in Section~\ref{rep} below, this representation is unfaithful in general when $G$ is infinite. The elements $s,t$ and $u$ have order $a,b$ and $c$ in $G$ (see the proof of Lemma~\ref{lem:conj_class} below). 

\begin{definition}\label{defn:reflection}
The \defn{reflections} in $H=J\begin{pmatrix} a & b & c \\ a' & b' & c' \end{pmatrix}\trianglelefteq J\begin{pmatrix} a & b & c \\ ~ & ~ & ~ \end{pmatrix}=G$ are those elements of $H$ which are $G$-conjugates of a power of $s^{a'},t^{b'}$, or $u^{c'}$ which is not the identity. We denote by $R(H)$ the set of reflections in $H$.  
\end{definition} 

\begin{definition}
Let $G, G'$ be two $J$-groups with respective sets of reflections $R, R'$. Let $\varphi: G\longrightarrow G'$ be a group homomorphism. We say that $\varphi$ is a \defn{reflection homomorphism}\footnote{Another possible definition of \defn{reflection homomorphism} could be to require that $\varphi(R)\subseteq R'\cup\{1\}$, which might be more suitable to study quotients as done for instance in~\cite{maxwell} in the case of finite Coxeter groups and their standard generators. In this paper we shall only consider reflection isomorphisms, which are obviously the same in both categories.} if $\varphi(R)\subseteq R'$. We say that $G$ and $G'$ are \defn{reflection isomorphic} if there is a reflection homomorphism $\varphi:G\longrightarrow G'$ which is a group isomorphism and such that 
$\varphi(R)=R'$. 
\end{definition}

\begin{rmq}\label{rmq_perm} Note that permuting the columns of a $J$-group $G$ yields another $J$-group $G'$ which is reflection isomorphic to $G$: a reflection isomorphism between $J\begin{pmatrix} a & b & c\\ a'& b'& c'\end{pmatrix}$ and $J\begin{pmatrix} b & a & c\\ b'& a'& c'\end{pmatrix}$ is obtained by mapping $s$ to $t^{-1}$, $t$ to $s^{-1}$, and $u$ to $u^{-1}$, while a reflection isomorphism between $J\begin{pmatrix} a & b & c\\ a'& b'& c'\end{pmatrix}$ and $J\begin{pmatrix} c & a & b\\ c'& a'& b'\end{pmatrix}$ is obtained by mapping $s$ to $t$, $t$ to $u$, and $u$ to $s$. These two reflection isomorphisms are enough to generate every permutation of the columns. 
\end{rmq}

\begin{lemma}\label{lem:conj_class}
No two elements in $\{s, s^2, \dots, s^{a-1}, t, t^2, \dots, t^{b-1}, u, u^2, \dots, u^{c-1}\}$ are conjugate to each other in $G=J\begin{pmatrix} a & b & c\\ ~& ~& ~\end{pmatrix}$. In particular, there are $a+b+c-3$ conjugacy classes of reflections in $G$.  
\end{lemma}

\begin{proof}
It suffices to note that the quotient of $J\begin{pmatrix} a & b & c\\ ~& ~& ~\end{pmatrix}$ by $J\begin{pmatrix} a & b & c\\ ~& b& c\end{pmatrix}$ is isomorphic to the abelian group $\mathbb{Z}/b\mathbb{Z}\times \mathbb{Z}/c\mathbb{Z}$, with $t$ having image $(\overline{1},\overline{0})$ and $u$ having image $(\overline{0},\overline{1})$. This shows that no two elements in $\{t, t^2, \dots, t^{b-1}, u, u^2, \dots, u^{c-1}\}$ are conjugate in $G$, as their images in the abelian quotient would have to be equal. Arguing similarly with $J\begin{pmatrix} a & b & c\\ a& b& ~\end{pmatrix}$ and $J\begin{pmatrix} a & b & c\\ a& ~& c\end{pmatrix}$ we get the claim. Note that this also shows that $s, t$ and $u$ have order $a, b$ and $c$ in $G$. 
\end{proof}

\begin{lemma}
With the notation of Definition~\ref{defn:reflection}, we have $R(H)=R(G)\cap H$. 
\end{lemma}

\begin{proof}
It is clear that $R(H)\subseteq R(G)\cap H$. Conversely, let $wr^n w^{-1}\in R(G)\cap H$, with $r\in \{s,t,u\}$ and $n$ not divisible by $o(r):=\mathrm{order}(r)$. As $H\trianglelefteq G$ we have $r^n\in H$. Assume without loss of generality that $r=s$. It remains to show that $a'$ divides $n$. If not, then there is $1\leq a''< a'$ such that $s^{a''}\in H$. But $$G / H \cong J\begin{pmatrix}a' & b' & c'\\ ~& ~& ~\end{pmatrix},$$ and since the image of $s$ in this quotient has order $a'$, one cannot have $s^{a''}\in H$, a contradiction. Hence $a'$ divides $n$, and $wr^n w^{-1}\in R(H)$. 
\end{proof}

\begin{definition}
Let $H$ be a $J$-group. Define $\H(H)$ to be the quotient set of $R(H)$ by the transitive closure of the relation $r^n\sim r^m$ for all $r\in R(H)$ and $1\leq n,m< o(r)$. An element $\H$ of $\H(H)$ will be called a \defn{reflecting hyperplane} (attached to any reflection $r\in R(H)$ such that the class $[r]$ of $r$ modulo~$\sim$ is equal to $\H$). Note that the action of $H$ on $R(H)$ by conjugation induces an action of $H$ on $\H(H)$---more generally, if $H$ has parent $J$-group $G$, then since $H$ is normal in $G$ we have an action of $G$ on $\H(H)$.  
\end{definition}

\begin{lemma}\label{cor:j_groups_char}
If a $J$-group $H\neq 1$ has a single $H$-conjugacy class of reflecting hyperplanes, then up to permutation of the columns $H$ is of the form $J\begin{pmatrix} a & b & c\\ a'& b& c\end{pmatrix}$, where $a'\neq a$, $b < c$. 
\end{lemma}

\begin{proof}
Given a $J$-group $H=J\begin{pmatrix} a & b & c\\ a'& b'& c'\end{pmatrix}$, if $i'<i$ for more than one integer $i\in\{a,b,c\}$, say $a$ and $b$, then both $1\neq s^{a'}$ and $1\neq t^{b'}$ lie in $H\trianglelefteq G=J\begin{pmatrix} a & b & c\\ ~& ~& ~\end{pmatrix}$. For $r\in R(H)$, by Lemma~\ref{lem:conj_class}, the reflection $r$ is conjugate (in $G$) to exactly one non-trivial power of a generator in $\{s,t,u\}$ and hence, all reflections $r^m$ for $1\leq m <o(r)$ are conjugate in $G$ to a power of the same generator. Now for any $r\in R(H)$ such that $[r]$ lies in the $H$-conjugacy class of $[s^{a'}]$, the reflection $r$ is conjugate (in $G$) to a power of $s$. Similarly, for any $r\in R(H)$ such that $[r]$ lies in the $H$-conjugacy class of $[t^{b'}]$, the reflection $r$ is conjugate (in $G$) to a power of $t$. It follows that the $H$-conjugacy classes of $[t^{b'}]$ and $[s^{a'}]$ are distinct in $G$ (and a fortiori in $H$), since if $[r]$ was a hyperplane in both classes, then $r$ would be conjugate (in $G$) to both a power of $s$ and a power of $t$, which by Lemma~\ref{lem:conj_class} is excluded.

Hence, we have that $H$ is of the form $J\begin{pmatrix} a & b & c\\ a'& b& c\end{pmatrix}$ where $a'<a$. Note that $b\neq c$ as $b'\neq c'$ by definition. Using also Remark~\ref{rmq_perm} this concludes the proof.
\end{proof}

Note that it is not clear a priori that every $J$-group of the form $H=J\begin{pmatrix} a & b & c\\ a'& b& c\end{pmatrix}$ as above has a single conjugacy class of reflecting hyperplanes: by definition, the generators of $H$ are all conjugate \emph{in $G$} to a power of $s$, but it is not clear that they are in fact conjugate in $H$ to a power of $s$. We shall show in Corollary~\ref{coro_single_class} below that this holds at least for $a'=1$, using Theorem~\ref{part_I_main} and the following result: 

\begin{prop}\label{h_conjugacy}
Let $H=J\begin{pmatrix} a & b & c\\ a'& b& c\end{pmatrix}$, where $a'\neq a$. \begin{enumerate}
\item The reflections $x_i:=t^{i-1} s^{a'} t^{-i+1}$, where $i=1,2, \dots, a$, generate $H$ (as a group),
\item Every reflection in $H$ is a conjugate in $H$ of a non-trivial power of one of the $x_i$'s.  
\end{enumerate}
\end{prop}

\begin{proof}
Since all the $x_i$'s lie in $H$ and $H\trianglelefteq G=J\begin{pmatrix} a & b & c\\ ~& ~& ~\end{pmatrix}$, their conjugates by $s, t$ and $u$ again lie in $H$. We begin by showing that an element of the form $g x_i g^{-1}$ where $g\in G$ can be rewritten in the form $h x_j h^{-1}$ for some $j$ and some element $h$ lying in the subgroup $H'\subseteq H$ generated by the $x_i$'s. Note that by induction on the length of a word for $g$ in $\{s^{\pm 1}, t^{\pm 1}, u^{\pm 1}\}$, it is enough to show it for $g\in\{s^{\pm 1},t^{\pm 1},u^{\pm 1}\}$. For $g=s^{\pm 1}$ we have $s^{\pm 1} x_i s^{\mp 1}=x_1^{\pm 1} x_i x_1^{\mp 1}$ as $s=x_1$. We have $t^{\pm 1} x_i t^{\mp{1}}=x_{i\pm{1}}$ for all $i=0, \dots, a-1$ (with the convention that $x_{a+1}=x_1$ and $x_0=x_a$). Using that $stu$ is central in $G$, we have $$u x_i u^{-1}= t^{-1} s^{-1} (s t u) x_i (u^{-1} t^{-1} s^{-1}) st =t^{-1} s^{-1} x_i s t= t^{-1} x_1^{-1} x_i x_1 t=x_{a}^{-1} x_{i-1} x_a,$$ and a similar computation shows that $u^{-1} x_i u=x_1 x_{i+1} x_1^{-1}$. 

Now $H$ is generated by its reflections, that is, by the conjugates in $G$ of the non-trivial powers of $s^{a'}=x_1$. But by the observation above, every element of the form $g x_1^\ell g^{-1}$ can be rewritten in the form $h x_i^{\ell} h^{-1}$ for some $i$ and some $h$ in $H'$. It follows that the $x_i$'s generate $H$, hence that $H'=H$. The second point also follows.  
\end{proof}

\subsection{Toric reflection groups are $J$-groups}\label{tarej}

Let $k,n,m\geq 2$ with $n<m$ and $n,m$ coprime. Recall from the introduction that the \defn{toric reflection group} $W(k,n,m)$ is defined by the presentation \begin{align}\label{pres_wknm_2} W(k,n,m):=\bigg\langle \ x_1, x_2, \dots, x_n \ \bigg\vert\ 
\begin{matrix} x_i^k=1~\text{for}~i=1, \dots, n,\\ \underbrace{x_1 x_2 \cdots}_{m~\text{factors}} = \underbrace{x_2 x_3\cdots}_{m~\text{factors}} = \dots = \underbrace{x_n x_1 \cdots}_{m~\text{factors}} \end{matrix} 
\ \bigg\rangle.\end{align}
In particular, the group $W(k,n,m)$ is a quotient of the torus knot group $G(n,m)$ as the above presentation is obtained from the classical presentation~\eqref{pres_2} of $G(n,m)$ by adding the relations $x_i^k=1$. 

\begin{definition}
The set of \defn{reflections} in $W(k,n,m)$, denoted $R$, is the set of those elements in $W(k,n,m)$ which are conjugate to a non-trivial power of one of the $x_i$'s, that is, $$R=\{ g x_i^{\ell} g^{-1}~|~i\in\{1, 2, \dots, n\},~g\in W(k,n,m), ~\ell\in\mathbb{Z}~\text{with}~x_i^{\ell}\neq 1\}.$$ We shall see later that all the $x_i$'s have order $k$. Given two toric reflection groups $W, W'$ with respective sets of reflections $R, R'$, we say that a group isomorphism $\varphi : W\longrightarrow W'$ is a \defn{reflection isomorphism} if $\varphi(R)=R'$, and write $W\cong_{\mathrm{ref}} W'$. 
\end{definition}

\begin{rmq}
Observe that Presentation~\eqref{pres_wknm_2} could be given as well for $m<n$. We explain why we will assume that $n<m$ in most statements. In fact, one has $G(n,m)\cong G(m,n)$ since the torus knots $T_{n,m}$ and $T_{m,n}$ are isotopic (see for instance~\cite[Proposition 3.37]{Knots}). By writing down explicitly an isomorphism between $G(n,m)$ and $G(m,n)$, it is not hard to see that it induces a reflection isomorphism $W(k,n,m)\cong W(k,m,n)$. We shall see in this section another explanation for this isomorphism, by realizing toric reflection groups as $J$-groups. More precisely, this will be a consequence of the first point of Corollary~\ref{coro_single_class} below together with the reflection isomorphism $$J\begin{pmatrix} k & n & m\\ ~& n& m\end{pmatrix}\cong_{\mathrm{ref}} J\begin{pmatrix} k & m & n\\ ~& m& n\end{pmatrix}.$$ from Remark~\ref{rmq_perm}.   
\end{rmq}

\begin{theorem}[Toric reflection groups are $J$-groups]\label{part_I_main}
Let $k, n, m\geq 0$ with $n,m$ coprime (we do not necessarily assume $n<m$ here), and let $H=J\begin{pmatrix} k & n & m\\ ~& n& m\end{pmatrix}\trianglelefteq J\begin{pmatrix} k & n & m\\ ~& ~& ~\end{pmatrix}=G$. Then $H$ has a presentation with generators $x_1, x_2, \dots, x_n$ and relations (indices are taken modulo $n$) \begin{align*} &x_i^k=1,~\forall i=1, \dots, n,\\
&x_1 x_2 \cdots x_m = x_i x_{i+1} \cdots x_{i+m-1}, ~\forall i=2, \dots, n.  
\end{align*}
If $n<m$ we therefore have $W(k,n,m)\cong H$. In terms of the generators of $G$ we have $x_i=t^{i-1} s t^{-i+1}$ for all $i=1, \dots, n$. 
\end{theorem}    

\begin{cor}\label{coro_single_class}
Let $H=J\begin{pmatrix} k & n & m\\ ~& n& m\end{pmatrix}$, with $k,n,m$ as in Theorem~\ref{part_I_main}. 
\begin{enumerate}
\item Let $R$ be the set of reflections in $W(k,n,m)$ and $R'$ be the set of reflections in $H=J\begin{pmatrix} k & n & m\\ ~& n& m\end{pmatrix}$. The isomorphism $\varphi : W(k,n,m)\overset{\cong}{\longrightarrow} H, x_i \mapsto t^{i-1} s t^{-i+1}$ from Theorem~\ref{part_I_main} satisfies $\varphi(R)=R'$. In other words, it is an isomorphism of "reflection groups". 
\item There is a single $H$-conjugacy class of reflecting hyperplanes in $H$. 
\end{enumerate}
\end{cor}

\begin{proof}
All the elements $t^{i-1} s t^{-i+1}$ are reflections in $H$, hence we have $\varphi(R)\subseteq R'$. Conversely, let $r$ be a reflection in $H$. Then by the second point of Proposition~\ref{h_conjugacy}, it is conjugate in $H$ to a non-trivial power of the $x_i$'s, hence $\varphi^{-1}(R')\subseteq R$. 

By definition of $H$, we know that all the reflecting hyperplanes in $H$ are conjugate in $G=J\begin{pmatrix} k & n & m\\ ~& ~& ~\end{pmatrix}$, but it is not obvious that they are conjugate in $H$. But by the second point of Proposition~\ref{h_conjugacy}, we know that every reflection in $H$ is conjugate in $H$ to a non-trivial power of some $x_i$. To conclude the proof, it therefore suffices to show that all the $x_i$'s are conjugate in $H$. We use Theorem~\ref{part_I_main} to this end: since $n$ and $m$ are coprime, all $x_i$'s are conjugate to each other in $W(k,n,m)$, as a consequence of the relations $$\underbrace{x_1 x_2 \cdots}_{m~\text{factors}} = \underbrace{x_2 x_3\cdots}_{m~\text{factors}} = \dots = \underbrace{x_n x_1 \cdots}_{m~\text{factors}}.$$   
\end{proof}

The proof of Theorem~\ref{part_I_main} will occupy the remainder of the section. It will be obtained via an application of the Reidemeister-Schreier algorithm (see for instance~\cite{MKS}). 

\begin{exple}
For $k=2, n=3, m=4$, we get the presentation $$x_i^2=1\textit{~for all~}i, \ x_1 x_2 x_3 x_1=x_2 x_3 x_1 x_2 = x_3 x_1 x_2 x_3.$$ This is a presentation of the complex reflection group $G_{12}$ in Shephard-Todd notation (see for instance~\cite[Table A.3]{Broue}), and we recover (as in~\cite{AA}) that $J\begin{pmatrix} 2 & 3 & 4\\ ~& 3& 4\end{pmatrix}\cong G_{12}$. Similarly, for the values of $k,n,m$ given in Table~\ref{table_finite_toric} below we obtain all the finite toric reflection groups (the fact that they are the only finite toric groups follows from the classification given in~\cite{AA}). These are the finite complex reflection groups of rank two with a single conjugacy class of reflecting hyperplanes: 

\medskip
\begin{table}[h!]
\begin{center}
\begin{tabular}{|c|c|c|c|}
\hline
$k$ & $n$ & $m$ & $W(k,n,m)$\\
\thickhline
$2$ & $3$ & $4$ & $G_{12}$\\
\hline
$2$ & $3$ & $5$ & $G_{22}$\\
\hline
$3$ & $2$ & $3$ & $G_{4}$\\
\hline
$4$ & $2$ & $3$ & $G_{8}$\\
\hline
$5$ & $2$ & $3$ & $G_{16}$\\
\hline
$3$ & $2$ & $5$ & $G_{20}$\\
\hline
$2$ & $2$ & $\geq 3$ and odd & $G(m,m,2)=I_2(m)$\\
\hline
\end{tabular}
\end{center}
\caption{Finite toric reflection groups.\label{table_finite_toric}}
\end{table}
\end{exple}

\begin{rmq}
It is well-known that $G_{16}$, $G_{20}$ and $G_{22}$ are normal subgroups of $G_{19}=J\begin{pmatrix} 2 & 3 & 5\\ ~& ~& ~\end{pmatrix}$, that $G_4$ is a normal subgroup of $G_7=J\begin{pmatrix} 2 & 3 & 3\\ ~& ~& ~\end{pmatrix}$, and that $G_8$ and $G_{12}$ are normal subgroups of $G_{11}=J\begin{pmatrix} 2 & 3 & 4\\ ~& ~& ~\end{pmatrix}$ (see for instance~\cite[Chapter 6]{LT}). In particular, Theorem~\ref{part_I_main} reproves this fact and gives a way to express the generators of these subgroups in terms of the generators of $G_{19}$, $G_{7}$ and $G_{11}$ respectively (see also~\cite[Table 1]{MM}).  
\end{rmq}

Let $G, H$ be as in Theorem~\ref{part_I_main}. Note that, as $G/H\cong \mathbb{Z}/n\mathbb{Z} \times \mathbb{Z}/m\mathbb{Z}$, a system of representatives of the (right) cosets modulo $H$ is given by $H u^i t^j$, $0\leq i < m$, $0\leq j < n$. Note that the set $\mathcal{K}:=\{ u^i t^j\}_{0\leq i < m, 0 \leq j < n}$ yields a Schreier transversal for $G/H$. For $g\in G$, we denote by $\overline{g}$ the representative of $Hg$ in $\mathcal{K}$. Recall that a system of generators for $H$ is given by the elements of the form $k x \overline{kx}^{-1}$ where $k\in\mathcal{K}$ and $x\in\{s,t,u\}$. For $x\in\{s,t,u\}$ and $0\leq i < m$, $0\leq j < n$, we denote the generator $u^i t^j x \overline{u^i t^j x}^{-1}$ of $H$ by $x_{i,j}$. Among this set of generators, some of them are the identity, namely 
\begin{itemize}
\item $t_{i,j}=1$ for all $0\leq i < m$, $0\leq j < n$, 
\item $u_{i,0}=1$ for all $0\leq i < m$. 
\end{itemize}  
For simplicity of notation in the next proofs, it will be convenient to consider the first subscript in $x_{i,j}$ modulo $m$ and the second one modulo $n$, that is, to write $x_{i+m,j}=x_{i,j}$ and $x_{i,j+n}=x_{i,j}$ for all $i,j\in\mathbb{Z}$.  

The graph given by the action of generators of $G$ on (right) cosets modulo $H$ is given in Figure~\ref{fig_1} below. More precisely, if a generator $x\in\{s,t,u\}$ is such that $(H u^i t^j)x=H u^{i'} t^{j'}$, then we draw an arrow from the vertex $H u^i t^j$ to the vertex $H u^{i'} t^{j'}$, indexed by $x_{i,j}$. The plain arrows give the spanning tree corresponding to the Schreier transversal with respect to the generating set $\{s,t,u\}$ of $G$. In this case the corresponding generator $x$ is equal to $1$, and we simply label the corresponding arrow by $x$. The reason for such a notation is that it becomes easy using this graph to write down the relations given by the Reidemeister-Schreier algorithm: one simply applies the defining relations of $G$ at each vertex of the coset graph, following arrows from left to right: for instance, applying the relations $stu=tus=ust$ at the vertex $Hut$ (see Figure~\ref{ut}) yields the relations $$s_{1,1} u_{1,2}= u_{1,2} s_{2,2}=u_{1,1} s_{2,1}$$ as arrows labelled by $t$ have their corresponding generator equal to $1$.

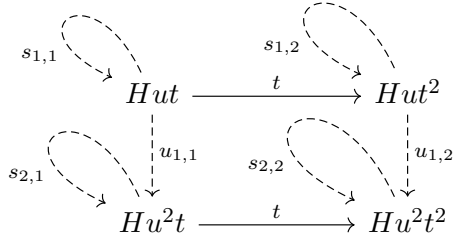
\begin{figure}[h!]
\begin{tikzcd} Hut\arrow[loop, dashed, out=115, in=155, looseness=10, "s_{1,1}"' near end]\arrow[r, "t"]\arrow[d, dashed, "u_{1,1}"]& [3em] Hut^2 \arrow[loop, dashed, out=115, in=155, looseness=10, "s_{1,2}"' near end]\arrow[d, dashed, "u_{1,2}"] \\ [2.5ex]
Hu^2t\arrow[loop, dashed, out=115, in=155, looseness=10, "s_{2,1}"' near end]\arrow[r, "t"]& Hu^2t^2 \arrow[loop, dashed, out=115, in=155, looseness=10, "s_{2,2}"' near end]
\end{tikzcd}
\caption{Application of the relations $stu=tus=ust$ at the vertex $Hut$ of the coset graph to obtain defining relations of $H$.\label{ut}} 
\end{figure}

\begin{figure}

\adjustbox{scale=0.95, center}{
\begin{tikzcd}H\arrow[r, "t"]\arrow[d, "u"]\arrow[loop above, dashed, "s_{0,0}"]& [2.5em] Ht\arrow[r,"t"]\arrow[d, dashed, "u_{0,1}"' near start]\arrow[loop above, dashed, "s_{0,1}"]& [2.5em] Ht^2 \arrow[r, "t"]\arrow[d, dashed, "u_{0,2}"]\arrow[loop above, dashed, "s_{0,2}"]& [0.5em] \dots \arrow[r,"t"]\arrow[d, dashed]& [1em] Ht^{n-1}\arrow[d, dashed, "u_{0,n-1}"]\arrow[loop above, dashed, "s_{0,n-1}"] \\ [2.5ex]
Hu\arrow[loop, dashed, out=120, in=160, looseness=10, "s_{1,0}"']\arrow[r, "t"]\arrow[d, "u"]& Hut\arrow[loop, dashed, out=115, in=155, looseness=10, "s_{1,1}"' near end]\arrow[r, "t"]\arrow[d, dashed, "u_{1,1}"]& Hut^2 \arrow[loop, dashed, out=115, in=155, looseness=10, "s_{1,2}"' near end]\arrow[r, "t"]\arrow[d, dashed, "u_{1,2}"]& \dots \arrow[r, "t"]\arrow[d, dashed]& Hut^{n-1}\arrow[loop, dashed, out=115, in=155, looseness=8, "s_{1,n-1}"' near end]\arrow[d, dashed, "u_{1,n-1}"] \\ [2.5ex]
Hu^2\arrow[loop, dashed, out=115, in=155, looseness=10, "s_{2,0}"']\arrow[r, "t"]\arrow[d, "u"]& Hu^2t\arrow[loop, dashed, out=115, in=155, looseness=10, "s_{2,1}"' near end]\arrow[r, "t"]\arrow[d, dashed]& Hu^2t^2 \arrow[loop, dashed, out=115, in=155, looseness=10, "s_{2,2}"' near end]\arrow[r, "t"]\arrow[d, dashed]& \dots \arrow[r, "t"]\arrow[d, dashed]& Hu^2t^{n-1}\arrow[d, dashed]\arrow[llll, dashed, out=194, in=346, color=red, "t_{2,n-1}=1" description] \\ [2.5ex]
\dots\arrow[r, "t"]\arrow[d, "u"]& \dots\arrow[r, "t"]\arrow[d, dashed]& \dots\arrow[r, "t"]\arrow[d, dashed]& \dots \arrow[r, "t"]\arrow[d, dashed]& \dots\arrow[d, dashed] \\ [2.5ex]
Hu^{m-2}\arrow[r, "t"]\arrow[d, "u"]& Hu^{m-2}t\arrow[r, "t"]\arrow[d, dashed, "u_{m-2,1}"']& Hu^{m-2}t^2 \arrow[r, "t"]\arrow[d, dashed, "u_{m-2,2}"']& \dots \arrow[r, "t"]\arrow[d, dashed]& Hu^{m-2}t^{n-1}\arrow[d, dashed, "u_{m-2,n-1}"'] \\ [2.5ex]
Hu^{m-1}\arrow[r, "t"]& Hu^{m-1}t\arrow[uuuuu, dashed, out=60, in=300, color=blue, "u_{m-1,1}"' description]\arrow[r, "t"]& Hu^{m-1}t^2 \arrow[r, "t"]& \dots \arrow[r, "t"]& Hu^{m-1}t^{n-1}
\end{tikzcd}
}
\caption{Coset graph of the (right) cosets of $H$ with action of the generators of $G$, and corresponding generators of $H$. For clarity we did not represent all the arrows, especially among those with label $t_{i,n-1}=1$, $u_{m-1, i}$, $s_{i,j}$. The plain arrows correspond to the spanning tree of the graph yielding the Schreier transversal (and have corresponding generator of $H$ equal to $1$).\label{fig_1}} 
\end{figure}
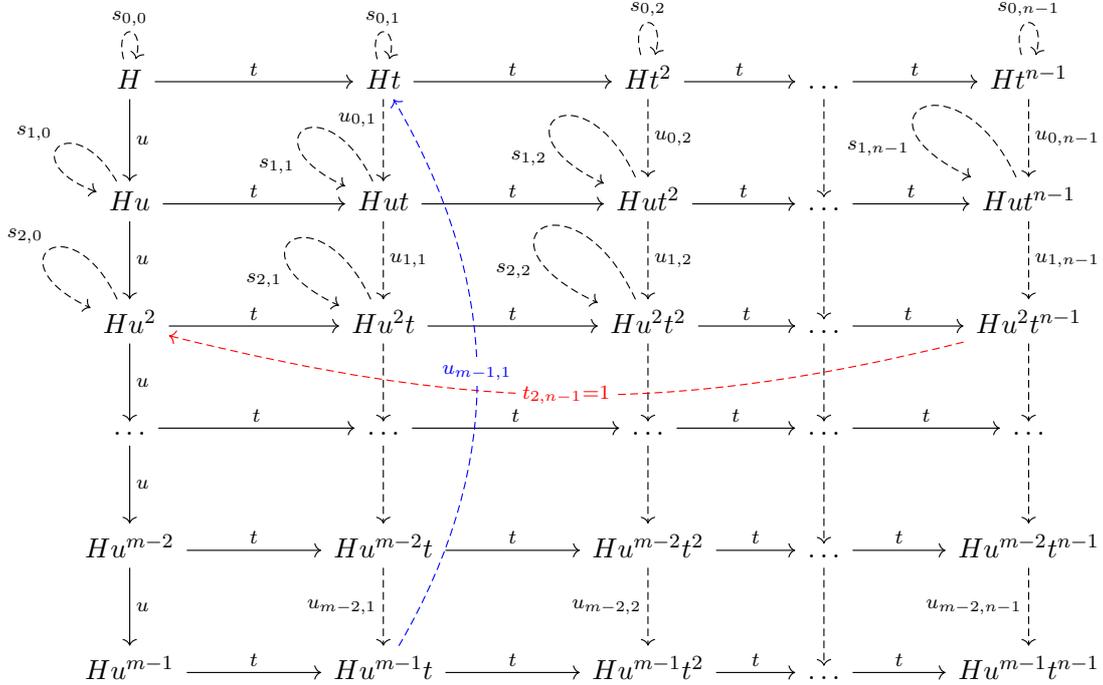   

We get the following relations: 

\medskip

\noindent\textbf{Relations coming from the relations $s^k=t^n=u^m=1$}. We get \begin{eqnarray}
s_{i,j}^k=1, \ \forall (i,j)\in\{0, \dots, m-1\}\times\{0, \dots, n-1\},\label{order_1}\\
u_{i,j} u_{i+1,j} \cdots u_{i+m-1,j}=1,\ \forall i\in\{0, \dots, m-1\}, j\in\{1,\dots, n-1\}.\label{order_4} 
\end{eqnarray}
\medskip
\noindent\textbf{Relations coming from the relation $stu=tus=ust$}.

\noindent\underline{First column of the graph in Figure~\ref{fig_1}:}
\begin{eqnarray}\label{first}
s_{i,0}u_{i,1}=u_{i,1} s_{i+1,1}= s_{i+1,0}, \ \forall i\in\{0, \dots, m-1\}.
\end{eqnarray}

\medskip

\noindent\underline{Column $j$, where $2\leq j \leq n-1$:}
\begin{eqnarray}\label{second}
s_{i,j-1}u_{i,j}=u_{i,j}s_{i+1,j}=u_{i,j-1}s_{i+1,j-1},\ \forall i\in\{0, \dots, m-1\}.
\end{eqnarray}

\noindent\underline{Column $n$:}
\begin{eqnarray}\label{third}
s_{i,n-1}=s_{i+1,0}=u_{i,n-1}s_{i+1,n-1},\ \forall i\in\{0, \dots, m-1\}.
\end{eqnarray}

We have already seen in Proposition~\ref{h_conjugacy} that the elements $x_i:=t^{i-1}st^{1-i}$ ($1\leq i \leq n$) generate $H$. Setting $s_i:=s_{0,i}$ ($1\leq i\leq n$) and considering indices of the $s_i$'s modulo $n$, for all $1\leq i \leq n$ we have \begin{align}\label{xist} s_{i-1}=s_{0,i-1}=t^{i-1} s \overline{(t^{i-1}s)}^{-1}=t^{i-1} s t^{1-i}=x_i\end{align} as $H t^{i-1}s=H \underbrace{t^{i-1}s t^{1-i}}_{\in H} t^{i-1}=H t^{i-1}$. The next proposition below therefore gives a new proof that the elements $x_1=s_0, x_2=s_1, \dots, x_n=s_{n-1}$ are enough to generate $H$, and gives a formula for all the elements $x_{i,j}$ of $H$ in terms of these generators.   

\begin{prop}\label{prop_generators}
Let $0\leq \ell \leq m-1$. \begin{enumerate}
\item Let $1\leq p \leq n$. We have $s_{m-1-\ell, p-1}=s_0 s_1 \cdots s_{\ell} s_{p+\ell} s_{\ell}^{-1} s_{\ell-1}^{-1} \cdots s_0^{-1}$.
\item Let $1\leq p \leq n-1$. We have $u_{m-1-\ell,p}=s_0 s_1 \cdots s_{\ell} s_{p+\ell}^{-1} s_{\ell-1}^{-1} s_{\ell-2}^{-1} \cdots s_0^{-1}$.
In particular, the $s_i$'s generate $H$. 
\end{enumerate}
\end{prop}

The proof of Proposition~\ref{prop_generators} will be by induction on $\ell$. The Lemma below deals with the case $\ell=0$. 

\begin{lemma}\label{lem:l0}
\begin{enumerate}
\item Let $1\leq p \leq n$. We have $s_{m-1, p-1}=s_0 s_{p} s_0^{-1}$.
\item Let $1\leq p \leq n-1$. We have $u_{m-1, p}=s_0 s_p^{-1}$.
\end{enumerate}
\end{lemma}

\begin{proof}
We argue by induction on $p$ to show both statements for $1\leq p \leq n-1$. Relation~\eqref{first} (with $i=m-1$) yields $$s_{m-1,0} u_{m-1,1}=u_{m-1,1} s_{1}=s_0,$$ from what we deduce that $u_{m-1,1}=s_0 s_1^{-1}$ and $s_{m-1,1}= s_0 s_1 s_0^{-1}$. Hence the statement of the lemma holds true for $p=1$.

Assume that the result holds true for some $1\leq p \leq n-2$. Relation~\eqref{second} (with $j=p+1$ and $i=m-1$) yields $$s_{m-1,p}u_{m-1,p+1}=u_{m-1,p+1} s_{p+1}=u_{m-1,p} s_{p},$$ from what by induction we get $u_{m-1,p+1}=u_{m-1,p} s_p s_{p+1}^{-1}= s_0 s_p^{-1}  s_p s_{p+1}^{-1}=s_0 s_{p+1}^{-1},$ and $s_{m-1,p}=u_{m-1,p+1} s_{p+1}u_{m-1,p+1}^{-1}=s_0 s_{p+1} s_0^{-1},$ hence both statements also hold true for $p+1$. Hence the result holds true for all $1\leq p \leq n-1$ and it remains to check the first statement for $p=n$, that is, that $s_{m-1,n-1}=s_0 s_{n} s_0^{-1}=s_0$. But this holds true as an immediate consequence of Relation~\eqref{third} with $i=m-1$. 
\end{proof}

\begin{proof}[Proof of Proposition~\ref{prop_generators}]
The proof is by induction on $\ell$. The case $\ell=0$ was treated in Lemma~\ref{lem:l0} above. Assume that the claimed formulas hold true for some $0\leq \ell \leq m-2$.

We show that the claimed relations hold true for $\ell+1$ by induction on $p$ as in Lemma~\ref{lem:l0}. Relation~\eqref{first} (with $i=m-2-\ell$) yields \begin{equation}\label{eq_11} s_{m-2-\ell,0} u_{m-2-\ell,1}=u_{m-2-\ell,1} s_{m-1-\ell,1}=s_{m-1-\ell,0},\end{equation} from what by induction (on $\ell$) we get \begin{align*} u_{m-2-\ell,1}&=s_{m-1-\ell,0}s_{m-1-\ell,1}^{-1}=(s_0 s_1 \cdots s_{\ell} s_{\ell+1} s_{\ell}^{-1} \cdots s_0^{-1})(s_0 s_1 \cdots s_{\ell} s_{\ell+2} s_{\ell}^{-1} \cdots s_0^{-1})\\&=s_0 s_1 \cdots s_{\ell} s_{\ell+1} s_{\ell+2}^{-1} s_{\ell}^{-1} \cdots s_0^{-1}.\end{align*}

Using this together with~\eqref{eq_11} again we get also by induction (on $\ell$) that \begin{align*} s_{m-2-\ell,0}&=s_{m-1-\ell,0}u_{m-2-\ell,1}^{-1}\\ &=(s_0 s_1 \cdots s_{\ell} s_{\ell+1} s_{\ell}^{-1} \cdots s_0^{-1})(s_0 s_1 \cdots s_{\ell} s_{\ell+2} s_{\ell+1}^{-1} s_{\ell}^{-1} \cdots s_0^{-1})\\ &=s_0 s_1 \cdots s_{\ell} s_{\ell+1} s_{\ell+2} s_{\ell+1}^{-1} s_{\ell}^{-1} \cdots s_0^{-1}  \end{align*} 
 
Hence the statement holds true for $p=1$.

Assume that the result holds true for some $1\leq p \leq n-2$. Relation~\eqref{second} (with $j=p+1$ and $i=m-2-\ell$) yields \begin{equation}\label{eq_2} s_{m-2-\ell,p}u_{m-2-\ell,p+1}=u_{m-2-\ell,p+1} s_{m-1-\ell,p+1}=u_{m-2-\ell,p} s_{m-1-\ell,p},\end{equation} from what by induction (on $\ell$ and $p$) we get \begin{align*} u_{m-2-\ell,p+1}&=u_{m-2-\ell,p} s_{m-1-\ell,p}s_{m-1-\ell,p+1}^{-1}\\ &= (s_0 s_1 \cdots s_{\ell+1} s_{p+\ell+1}^{-1} s_{\ell}^{-1} s_{\ell-1}^{-1} \cdots s_0^{-1})(s_0 s_1 \cdots s_{\ell} s_{p+\ell+1} s_{\ell}^{-1} s_{\ell-1}^{-1} \cdots s_0^{-1})\\ &~~~~\cdot (s_0 s_1 \cdots s_{\ell} s_{p+\ell+2}^{-1} s_{\ell}^{-1} s_{\ell-1}^{-1} \cdots s_0^{-1})\\ &=s_0 s_1 \cdots s_{\ell+1}s_{p+\ell+2}^{-1} s_{\ell}^{-1} s_{\ell-1}^{-1} \cdots s_0^{-1}. \end{align*}

Using this together with~\eqref{eq_2} again we get also by induction (on $\ell$ and $p$) that \begin{align*} s_{m-2-\ell,p}&=u_{m-2-\ell,p} s_{m-1-\ell,p}u_{m-2-\ell,p+1}^{-1}\\ &= (s_0 s_1 \cdots s_{\ell+1} s_{p+\ell+1}^{-1} s_{\ell}^{-1} s_{\ell-1}^{-1} \cdots s_0^{-1})(s_0 s_1 \cdots s_{\ell} s_{p+\ell+1} s_{\ell}^{-1} s_{\ell-1}^{-1} \cdots s_0^{-1})\\ &~~~~\cdot (s_0 s_1 \cdots s_{\ell}s_{p+\ell+2} s_{\ell+1}^{-1}s_{\ell}^{-1} \cdots s_0^{-1})\\ &=s_0 s_1 \cdots s_{\ell+1}s_{p+\ell+2}^{-1} s_{\ell+1}^{-1}s_{\ell}^{-1} s_{\ell-1}^{-1} \cdots s_0^{-1},\end{align*}
hence both statements also hold true for $p+1$. Hence the result holds true at level $\ell+1$ for all $1\leq p \leq n-1$ and it remains to check the first statement (at level $\ell+1$) for $p=n$, that is, that $$s_{m-2-\ell,n-1}=s_0 s_1 \cdots s_{\ell+1} s_{n+1+\ell} s_{\ell+1}^{-1} s_{\ell}^{-1} \cdots s_0^{-1}=s_0 s_1 \cdots s_{\ell+1} s_{\ell}^{-1} \cdots s_0^{-1}.$$ To this end we use Relation~\eqref{third} with $i=m-2-\ell$, that is, $$s_{m-2-\ell,n-1}=s_{m-1-\ell,0}=u_{m-2-\ell,n-1} s_{m-1-\ell,n-1}.$$
But already know by induction (on $\ell$) that $s_{m-1-\ell,0}=s_0 s_1 \cdots s_{\ell+1} s_{\ell}^{-1} \cdots s_0^{-1}$ (alternatively both factors in the product $u_{m-2-\ell,n-1} s_{m-1-\ell,n-1}$ are also already known, and multiplying them yields the same result).
\end{proof}

Setting $\ell=m-1$ in the first item of Proposition~\ref{prop_generators} we get \begin{equation}\label{gen_brute}s_{i-1}=s_0 s_1 \cdots s_{m-1} s_{i+m-1} s_{m-1}^{-1} \cdots s_1^{-1} s_0^{-1},\ \forall 1\leq i \leq n.\end{equation}

\begin{lemma}\label{rel_equiv}
Setting $x_i=s_{i-1}$ for all $1\leq i \leq n$, the set of relations~\eqref{gen_brute} is equivalent to the second set of relations in the statement of Theorem~\ref{part_I_main}, that is, to $$x_1 x_2 \cdots x_m = x_i x_{i+1} \cdots x_{i+m-1}, ~\forall i=2, \dots, n.$$
\end{lemma}

\begin{proof}
The set of relations~\eqref{gen_brute} can be rewritten \begin{equation}\label{gen_brute_2}x_i x_1 x_2 \cdots x_m=x_1 x_2\cdots x_m x_{i+m},~\forall 1\leq i \leq n.\end{equation} We show that they imply the relations $x_1 x_2 \cdots x_m = x_j x_{j+1} \cdots x_{j+m-1}, ~\forall j=2, \dots, n$, by induction on $j$. Putting $i=1$ in~\eqref{gen_brute_2} and cancelling $x_1$ we get the relation $x_1 x_2\dots x_m= x_2 x_3\cdots x_{m+1}$, which proves the case $j=2$. Assume that $x_1 x_2 \cdots x_m = x_j x_{j+1} \cdots x_{j+m-1}$ for some $2\leq j < n$. From~\eqref{gen_brute_2} and by induction we get $$x_{j} x_1 x_2 \cdots x_m=x_1 x_2\cdots x_m x_{j+m}=x_j x_{j+1}\cdots x_{j+m-1} x_{j+m}.$$ Cancelling $x_j$ on both sides yields $x_1 x_2 \cdots x_m=x_{j+1}\cdots x_{j+m-1} x_{j+m}$. Hence the relations in~\eqref{gen_brute} imply the second set of relations. 

Conversely, if the second set of relations hold true, then for all $1 \leq i \leq n$ we have \begin{align*} x_i x_1 x_2 \cdots x_m&= x_i x_{i+1}\cdots x_{i+m-1} x_{i+m}=x_1 x_2\cdots x_m x_{i+m},\end{align*} which concludes the proof.
\end{proof}

\begin{proof}[End of the proof of Theorem~\ref{part_I_main}]
It follows from Proposition~\ref{prop_generators} that the $x_i$'s generate $H$. Lemma~\ref{rel_equiv} shows that the second set of relations in the statement of Theorem~\ref{part_I_main} hold true in $H$. We also know that the first set of relations hold true as a consequence of~\eqref{order_1}. To conclude the proof of Theorem~\ref{part_I_main}, it therefore remains to show that, replacing the various $u_{i,j}$ and $s_{i,j}$ in Relations~\eqref{first} to \eqref{third} by their expressions in terms of $s_0, s_1, \dots, s_{n-1}$ obtained in Proposition~\ref{prop_generators}, we get no other relations than those in the statement of Theorem~\ref{part_I_main}. 

This is a direct check. For Relations~\eqref{order_1} to~\eqref{order_4}, as by Proposition~\ref{prop_generators}, every $s_{i,j}$ is conjugate to some $s_i$, we get that Relations~\eqref{order_1} follow from the relations $x_i^k=1$ for all $1\leq i \leq n$. We now consider Relations~\eqref{order_4}. These relations tell us that, $\forall i\in\{0, m-1\}, j\in\{1, n-1\}$, we have $\prod_{q=i}^{i+m-1} u_{q,j}=1$. We separate this product as $\prod_{q=i}^{m-1} u_{q,j} \prod_{q=m}^{m-1+i} u_{q,j}$. Proposition~\ref{prop_generators} yields \begin{align*} \prod_{q=i}^{m-1} u_{q,j} =&\prod_{q=i}^{m-1} s_0 s_1 \cdots s_{m-1-q} s_{j+m-1-q}^{-1} s_{m-2-q}^{-1} s_{m-3-q}^{-1} \cdots s_0^{-1} \\& =s_0 s_1 \cdots s_{m-1-i} s_{j+m-1-i}^{-1} s_{j+m-2-i}^{-1} \cdots s_{j}^{-1},\end{align*} while \begin{align*} \prod_{q=m}^{m-1+i} u_{q,j}&=\prod_{q=0}^{i-1} u_{q,j}=s_0 s_1 \cdots s_{m-1-q} s_{j+m-1-q}^{-1} s_{m-2-q}^{-1} s_{m-3-q}^{-1} \cdots s_0^{-1}\\ &=s_0 s_1 \cdots s_{m-1} s_{j+m-1}^{-1}s_{j+m-2}^{-1}\cdots s_{j+m-i}^{-1} s_{m-i-1}^{-1} s_{m-i-2}^{-1}\cdots s_0^{-1}.\end{align*}

Now we have the relation $\prod_{q=i}^{m-1} u_{q,j} \prod_{q=m}^{m-1+i} u_{q,j}=1$. Replacing the expressions obtained above and conjugating by $(s_0 s_1 \cdots s_{m-1-j})^{-1}$ we get the relation $$s_{j+m-1-i}^{-1} s_{j+m-2-i}^{-1} \cdots s_{j}^{-1}s_0 s_1 \cdots s_{m-1} s_{j+m-1}^{-1}s_{j+m-2}^{-1}\cdots s_{j+m-i}^{-1}=1.$$ Putting the inverses in the right hand side and replacing $s_{i-1}$ by $x_i$ for all $i$ we get the equivalent relation $x_1 x_2 \cdots x_m=x_{j+1} x_{j+2} \cdots x_{j+m},$ which we already obtained in Lemma~\ref{rel_equiv} above. 

Hence Relations~\eqref{order_1} and~\eqref{order_4} yield no new relation in $H$. We need to check that the same holds true for the remaining Relations~\eqref{first} to~\eqref{third}. We check it for~\eqref{first}. For any $0\leq i \leq m-1$ we have \begin{align*} s_{i,0} u_{i,1}&=(s_0 s_1 \cdots s_{m-1-i} s_{m-i} s_{m-1-i}^{-1} \cdots s_0^{-1})(s_0 s_1 \cdots s_{m-1-i} s_{m-i}^{-1} s_{m-2-i}^{-1} \cdots s_0^{-1})\\ &=s_0 s_1 \cdots s_{m-1-i}s_{m-2-i}^{-1} \cdots s_0^{-1},\end{align*} \begin{align*} u_{i,1} s_{i+1,1}&=(s_0 s_1 \cdots s_{m-1-i} s_{m-i}^{-1} s_{m-2-i}^{-1} \cdots s_0^{-1})(s_0 s_1 \cdots s_{m-i-2} s_{m-i} s_{m-i-2}^{-1} \cdots s_0^{-1})\\ &=s_0 s_1 \cdots s_{m-1-i}s_{m-2-i}^{-1} \cdots s_0^{-1},\end{align*}
and we thus see that both $s_{i,0} u_{i,1}$ and $u_{i,1} s_{i+1,1}$ yield the same words and by Proposition~\ref{prop_generators} they are also equal to a word for $s_{i+1,0}$. We only deleted factors of the form $xx^{-1}$ to obtain these equalities, hence these relations do not yield any relation in $H$ at all.

It is checked that Relations~\eqref{second} and \eqref{third} do not give any new relation in $H$ in the exact same way as we did for~\eqref{first} above. For~\eqref{second}, replacing the various factors using Proposition~\ref{prop_generators} and deleting $xx^{-1}$ factors as above, one sees that $s_{i,j-1} u_{i,j}$, $u_{i,j} s_{i+1,j}$ and $u_{i,j-1}s_{i+1,j-1}$ again yield the same words, equal to the word for $s_{i+1,0}$ obtained in Proposition~\ref{prop_generators}. For~\eqref{third} we also get that $s_{i,n-1}$ and $u_{i,n-1} s_{i+1,n-1}$ yield the same words, also equal to the word for $s_{i+1,0}$ obtained in Proposition~\ref{prop_generators}. 

The last statement that $x_i$ corresponds in $G$ to $t^{i-1} s t^{1-i}$ has already been seen in~\eqref{xist} above.
\end{proof}

\subsection{Representation by complex reflection groups}\label{rep}

Achar and Aubert~\cite[Section 4]{AA} constructed a representation of $G=J\begin{pmatrix} a & b & c\\ ~& ~& ~\end{pmatrix}$ (and hence of any $J$-group) on $\mathbb{C}^2$, where the generators $s,t$ and $u$ act by (pseudo-)reflections. This yields reflection representations of $J$-groups and hence, thanks to Theorem~\ref{part_I_main}, of every toric reflection group. As we shall see in this section, this representation is \textit{not} faithful in general when $W(k,n,m)$ is infinite.

The representation is constructed as follows. Let $K$ be a finite abelian extension of $\mathbb{Q}$ containing the roots of unity of orders $2a, 2b,$ and $2c$. Let $\mathbb{Z}_{K}$ be the ring of algebraic integers of $K$. Let $\theta=e^{i\pi/a}$, $\phi=e^{i\pi/b}$, and $\psi=e^{i\pi/c}$. Choose two elements $q,r\in \mathbb{Z}_K$ such that $qr=\theta\phi(\psi+\psi^{-1})-\theta^2-\phi^2$. Then setting $$\rho(s)=\begin{pmatrix} \theta^2 & q \\ 0 & 1\end{pmatrix}, \ \rho(t)=\begin{pmatrix} 1 & 0\\ r & \phi^2 \end{pmatrix}, \ \rho(u)=\theta \phi\psi \rho(t)^{-1} \rho(s)^{-1}$$ we have

\begin{prop}[{\cite[Proposition 4.2]{AA}}] 
The map $\rho$ extends to a group homomorphism $G\longrightarrow\mathrm{GL}_2(K)$. 
\end{prop}

The following example shows that $\rho$ is unfaithful in general when $W(k,n,m)$ is infinite. 

\begin{exple}\label{ex_unfaith}
Let $k=6, n=2, m=3$. By Theorem~\ref{part_I_main}, the group $W(6,2,3)$ can be identified with the subgroup of $G=J\begin{pmatrix} 6 & 2 & 3\\ ~& ~& ~\end{pmatrix}$ generated by $x_1=s$ and $x_2=tst^{-1}=tst$. Note that, thanks to Theorem~\ref{part_I_main}, it is a quotient of the $3$-strand braid group $$B_3=\langle \ x_1, x_2 ~|~x_1 x_2 x_1=x_2 x_1 x_2 \ \rangle$$ by the relations $x_1^6=1=x_2^6$, studied by Coxeter in~\cite{coxeter_factor}. By definition of $\rho(u)$ and since $stu=tus=ust$, we have that $\rho(s)\rho(t)\rho(u)$ is the scalar matrix $\theta\phi\psi \mathrm{Id}$, which in this case is $-\mathrm{Id}$. It implies that $\rho(stu)$ has order $2$. Now since $stu$ is central in $G$ and $u$ has order $6$, we have that $(stu)^6=(stuu^{-1})^6=(st)^6$, from which we deduce that $$\rho(x_1 x_2)^3=\rho(stst)^3=\rho((stu)^6)=\mathrm{Id}.$$ On the other hand, there is a surjective map $W(6,2,3)\twoheadrightarrow W(3,2,3)$, $x_i\mapsto x_i$. The group $W(3,2,3)$ is the complex reflection group $G_4$, and one checks that $x_1 x_2$ has order $6$ in this group. It follows that $(x_1 x_2)^3$ cannot be equal to $1$ in $W(6,2,3)$, and that Achar and Aubert's representation is unfaithful in this case. The same observation can be made with the representation constructed by Coxeter in~\cite[Section 7]{coxeter_factor}. 

Moreover, the isomorphism class of the representation depends in fact on the choice of $q$ and $r$. In the above case, we have $qr=0$. Chosing $q=0=r$ yields matrices $\rho(s)$ and $\rho(t)$ commuting with each other, while choosing $q=1, r=0$ yields matrices $\rho(s)$ and $\rho(t)$ which do not commute with each other. 
\end{exple}

The above example raises the following questions:

\begin{question}
Are there examples of infinite toric reflection groups for which Achar and Aubert's representation is faithful? For which toric reflection groups $W(k,n,m)$ is there a (canonical) faithful representation as a complex reflection group?
\end{question}

\section{Center of toric reflection groups}\label{center}

The aim of this section is to establish that toric reflection groups have a cyclic center, and to show that the quotient by their center is an alternating subgroup of a Coxeter group of rank three. This will be a key ingredient for the classification of toric reflection groups which we shall give in the next section. We assume the reader to be familiar with the general theory of Coxeter groups and their parabolic subgroups (see for instance~\cite{Bourbaki, AB} for basics on the topic).

Let us introduce some notation and properties of Coxeter groups which will be used in the next sections. We will denote by $(W,S)$ a Coxeter system, with $S$ finite. Recall that for $I\subseteq S$, the subgroup $W_I:=\langle s~|~s\in I \rangle$ is called a \defn{standard parabolic subgroup} of $W$, and that the pair $(W_I, I)$ is itself a Coxeter system. These subgroups have particularly nice properties: for instance, any $S$-reduced decomposition of an element $w\in W$ which lies in $W_I$ has all its letters in $I$---see~\cite[Section 2.3.2]{AB}. A subgroup of the form $x W_I x^{-1}$ of $W$, where $I\subseteq S$ and $x\in W$, is called a \defn{parabolic subgroup} of $W$. 

More generally, one can show that any subgroup $W'$ of $W$ generated by a subset of the set $T=\bigcup_{w\in W} w S w^{-1}$ of \defn{reflections} of $W$ is itself a Coxeter group in a canonical way---see~\cite{Dyer_reflections}. Such a subgroup is called a \defn{reflection subgroup} of $W$. 

Given a Coxeter group $W$ with length function $\ell_S$ with respect of the generating set $S$, the subset $W^+$ of elements $w$ such that $\ell_S(w)$ is even forms a subgroup of index two of $W$ (hence normal), called the \defn{alternating subgroup} of $W$, which we shall denote $W^+$. 

\subsection{Center of alternating subgroups of Coxeter groups}

We will need the following result, which applies to an arbitrary Coxeter group of rank at least $3$, though we will only apply it in the rank three case. Recall that the center of an infinite and irreducible Coxeter group is trivial (see~\cite[Chap. V, Sec. 4, Exercise 3]{Bourbaki}).

\begin{prop}[Center of alternating subgroups of Coxeter groups]\label{prop_2}
Let $(W,S)$ be a Coxeter system with $|S| \geq 3$. Let $W^+$ be the alternating subgroup of $W$. Then $Z(W^+)\subseteq Z(W)$. In particular, as infinite and irreducible Coxeter groups have trivial center, the center of the alternating subgroup of an infinite and irreducible Coxeter group of rank at least three is trivial.    
\end{prop}

\begin{proof}
Let $x\in Z(W^+)$. Let $T$ denote the set $\bigcup_{w\in W} w S w^{-1}$ of reflections of $W$. For every $s\in S$, define $s':=xsx^{-1}\in T$. We will show that $s'=s$ for all $s\in S$, which concludes the proof. 

Let $s,t\in S$ with $s\neq t$. As $x$ is central in $W^+$, we have $xst=stx$. We also have $xst=s't'x$, which yields $st=s't'$. By~\cite[Lemma 3.1]{Dyer_Bruhat}, the reflection subgroup $W':=\langle s,t,s',t'\rangle$ of $W$ is dihedral. We claim that $W$ is the standard parabolic subgroup generated by $s$ and $t$. Indeed, writing $\chi(W')$ for its set of canonical generators as a Coxeter group (see~\cite{Dyer_reflections}), we have $|\chi(W')|=2$ and $s,t\in\chi(W')$ since $s,t\in S$, and hence $W'=\langle s, t\rangle$. In particular, we have $s',t'\in \langle s, t\rangle$. Assume that $s'\neq s$. Then since $s'\in\langle s,t\rangle$ and $\langle s, t\rangle$ is a standard parabolic subgroup of $W$, the reflection $s'$ has a reduced $S$-decomposition (equivalently all reduced $S$-decompositions) in which $t$ has to appear. Now, since $W$ has rank at least $3$, let $r\in S\backslash\{s,t\}$. Arguing as above we have $sr=s'r'$, and $s'\in\langle s,r\rangle$. This is a contradiction, as $t$ appears in reduced $S$-decompositions of $s'$, while all elements in $\langle s, r\rangle$ have all their $S$-reduced decompositions only involving the letters $s$ and $r$. Hence $s=s'$.  
\end{proof}

The above result is not valid for irreducible Coxeter groups of rank two: let $W=I_2(m)$ be the dihedral group of order $2m$ ($m\geq 3$). Then $W^+\cong C_m$, the cyclic group of order $m$. It is the rotation subgroup of $W$. Hence $W^+=Z(W^+)$, while $Z(W)$ is trivial for odd $m$ and isomorphic to $C_2$ for even $m$. The infinite dihedral group $W$ has trivial center, while $W^+\cong\mathbb{Z}$.  

Also note that it has been shown that for irreducible, infinite and non-affine Coxeter groups, the center of any finite index subgroup is trivial (see~\cite[Proposition 6.4]{Paris_irred} or~\cite{D_Q}). 

\subsection{Center of toric reflection groups and their parent $J$-groups}

Let $k,n,m\geq 2$ with $n<m$ and $n$ and $m$ coprime. The aim of this section is to show that the center of a toric reflection group is cyclic. A step to achieve this is to show that $W(k,n,m)/ Z(W(k,n,m))$ is isomorphic to the alternating subgroup $W_{k,n,m}^+$ of the rank-three Coxeter group  
\begin{align}\label{cox_3}
W_{k,n,m}=\bigg\langle r_1,r_2,r_3 \ \bigg\vert\ 
\begin{matrix}
r_1^2=r_2^2=r_3^2=1,\\
(r_1 r_2)^k=(r_2 r_3)^n=(r_3 r_1)^m=1
\end{matrix}
\ \bigg\rangle
\end{align}

A presentation for the alternating subgroup $W^+$ of an arbitrary Coxeter group $W$ is given in Bourbaki~\cite[Chap. IV, Sec. 1, Exercise 9]{Bourbaki}. In the specific case of $W_{k,n,m}$ it yields the presentation 
\begin{align}\label{alt_bou}
W_{k,n,m}^+=\bigg\langle a,b \ \bigg\vert\ 
\begin{matrix}
a^k=b^n=(ba^{-1})^m=1
\end{matrix}
\ \bigg\rangle,
\end{align}
where in terms of the generators $r_1, r_2, r_3$ of $W_{k,n,m}$ we have $a=r_1 r_2$, $b=r_3 r_2$ (hence $ba^{-1}=r_3 r_1$). 

\begin{lemma}\label{delta}
Let $m=qn+r$ be the Euclidean division of $m$ by $n$. Let $\delta:=x_1 x_2 \cdots x_m$. Then, taking indices modulo $n$, we have $x_i \delta = \delta x_{i+r}$ for all $i=1, \dots, n$.  
\end{lemma}

\begin{proof}
As $m\equiv r ~(\mathrm{mod}~n)$, using the defining relations of $W(k,n,m)$ we have for all $i=1, \dots, n$ $$x_i (x_1 x_2 x_3\cdots x_m)=x_i (x_{i+1} x_{i+2}\cdots x_{m+i})=(x_i x_{i+1} \cdots x_{m+i-1}) \underbrace{x_{m+i}}_{=x_{i+r}}=(x_1 x_2\cdots x_m) x_{i+r}.$$
\end{proof}

Let $c:=(x_1 x_2 \cdots x_n)^m\in W(k,n,m)$. Note that we have $(x_1 x_2 \cdots x_n)^m=(x_1 x_2 \cdots x_m)^n$. Moreover, by Lemma~\ref{delta}, the element $c$ is central in $W(k,n,m)$ as $$x_i c= x_i \delta^n=\delta^n x_{i+nr}=\delta^n x_i.$$ We denote by $\overline{W(k,n,m)}$ the quotient of $W(k,n,m)$ by the extra relation $c=1$.

We clearly have a group homomorphism $J\begin{pmatrix} k & n & m \\ ~ & ~ & ~ \end{pmatrix}\longrightarrow W_{k,n,m}^+$ with $s\mapsto r_1 r_2$, $t\mapsto r_2 r_3$, $u\mapsto r_3 r_1$. By restriction it induces a homomorphism $W(k,n,m)\longrightarrow W_{k,n,m}^+$. Considering the generators $x_i$ of $W(k,n,m)\cong J\begin{pmatrix} k & n & m\\ ~& n& m\end{pmatrix}$ and recalling from Theorem~\ref{part_I_main} that in terms of the generators of the parent $J$-group $J\begin{pmatrix} k & n & m\\ ~& ~& ~\end{pmatrix}$ we have $x_i=t^{i-1} s t^{1-i}$ for all $i=1, \dots, n$, we see that such a homomorphism maps $x_1 x_2 \cdots x_n$ to $(r_1 r_3)^n$, which has order $m$ in $W_{k,n,m}$: indeed, by the general theory of Coxeter groups we know that $r_1 r_3$ has order $m$ (see for instance~\cite[Section 2.3.3]{AB}), and $n$ and $m$ are coprime. In particular, the element $c$ is mapped to $1$, hence the above homomorphism factors through $\overline{W(k,n,m)}$, and we denote by $\varphi : \overline{W(k,n,m)}\longrightarrow W_{k,n,m}^+$ the obtained homomorphism.

The main result of the section is given by the following statement

\begin{theorem}[Center of toric reflection groups]\label{thm_alt}
Let $G=J\begin{pmatrix} k & n & m\\ ~& ~& ~\end{pmatrix}$, where, $k,n,m\geq 2$, $n<m$, and $n,m$ are coprime. 

\begin{enumerate} \item The map $\varphi$ is an isomorphism. \item We have the following commutative diagram, where both rows are short exact sequences

\smallskip 
\begin{tikzcd}
1 \arrow[r] & \langle stu \rangle\arrow[r] & G\arrow[r] & W_{k,n,m}^+ \arrow[r] & 1 \\
1 \arrow[r] & \langle c \rangle \arrow[r] \arrow[u, hook] & W(k,n,m)\arrow[r] \arrow[u, hook] & W_{k,n,m}^+\arrow[r] \arrow[u, "\mathrm{id}"] & 1 
\end{tikzcd}
\smallskip

\item We have $Z(W(k,n,m))=\langle c \rangle$ and $Z(G)=\langle stu \rangle$.

\item The above commutative diagram induces isomorphisms $$G/Z(G)\cong W(k,n,m)/ Z(W(k,n,m))=\overline{W(k,n,m)}\cong W_{k,n,m}^+.$$
\end{enumerate}
\end{theorem}

We will split the proof of Theorem~\ref{thm_alt} into several statements.

\begin{rmq}
As a byproduct, in the cases where $H=W(k,n,m)$ is finite, we recover from Theorem~\ref{thm_alt} the known (see for instance the tables at the end of~\cite{Broue}) description of $H/Z(H)$ given in Table~\ref{table_gzg}. Note that for the primitive groups, \emph{i.e.}, all the groups in Table~\ref{table_gzg} except dihedral groups, there are three possible groups $H/Z(H)$: this is a well-known fact, as these groups are of three types, called \textit{tetrahedral} ($G_4$), \textit{octahedral} ($G_8, G_{12}$), and \textit{icosahedral} ($G_{16}, G_{20}, G_{22}$), depending on whether $H$ is a subgroup of the tetrahedral group $\mathcal{T}$, octahedral group $\mathcal{O}$, or icosahedral group $\mathcal{I}$. If $H$ is of type $\mathcal{G}$ for $\mathcal{G}\in\{\mathcal{T}, \mathcal{O}, \mathcal{I}\}$, then $H/Z(H)\cong \mathcal{G}/ Z(\mathcal{G})$ (see~\cite[Chapter 6]{LT} for a detailed explanation of this phenomenon). Theorem~\ref{thm_alt} gives a more general, new description of $H/Z(H)$ by showing that it is an alternating subgroup of a Coxeter group of rank three which can be attached in a uniform way to all the concerned finite complex reflection groups. 
\begin{table}[h!]
\begin{center}
\begin{tabular}{|c|c|c|c|c|}
\hline
$k$ & $n$ & $m$ & ${W}(k,n,m)$ & $W(k,n,m)/Z(W(k,n,m))$ \\
\thickhline
$2$ & $3$ & $4$ & $G_{12}$ & $W_{2,3,4}^+=W(B_3)^+\cong\mathfrak{S}_4$\\
\hline
$2$ & $3$ & $5$ & $G_{22}$ & $W_{2,3,5}^+=W(H_3)^+\cong\mathfrak{A}_5$\\
\hline
$3$ & $2$ & $3$ & $G_{4}$ & $W_{3,2,3}^+=W(A_3)^+\cong\mathfrak{A}_4$\\
\hline
$4$ & $2$ & $3$ & $G_{8}$ & $W_{4,2,3}^+=W(B_3)^+\cong\mathfrak{S}_4$\\
\hline
$5$ & $2$ & $3$ & $G_{16}$ & $W_{5,2,3}^+=W(H_3)^+\cong\mathfrak{A}_5$ \\
\hline
$3$ & $2$ & $5$ & $G_{20}$ & $W_{3,2,5}^+=W(H_3)^+\cong \mathfrak{A}_5$ \\
\hline
$2$ & $2$ & $\geq 3$ and odd & $G(m,m,2)=I_2(m)$ & $W_{2,2,m}^+= W(A_1\times I_2(m))^+=G(m,m,2)$\\
\hline
\end{tabular}
\end{center}
\caption{Quotient of a finite toric reflection group by its center.\label{table_gzg}}

\end{table}

\end{rmq}

\begin{rmq} As another byproduct, Theorem~\ref{thm_alt} also gives new presentations for alternating subgroups of Coxeter groups of rank three whose Coxeter diagrams have two edges with coprime labels. That is, for $k,n,m\geq 2$, $n<m$, and $n,m$ are coprime, we have \begin{align*} 
W_{k,n,m}^+=\bigg\langle x_1,x_2,\dots, x_n \ \bigg\vert\ 
\begin{matrix}
x_i^k=1~\forall i=1, \dots, n,\\ x_1 x_2 \cdots x_m = x_i x_{i+1} \cdots x_{i+m-1}, ~\forall i=2, \dots, n, \\ (x_1 x_2 \cdots x_n)^m=1.
\end{matrix}
\ \bigg\rangle
\end{align*}
\end{rmq}

The following statement proves the first point of Theorem~\ref{thm_alt}:

\begin{prop}\label{prop_1}
The map $\varphi$ is an isomorphism. 
\end{prop}

\begin{proof}
We still denote the images of the generators $x_i$ of $W(k,n,m)$ in the quotient $\overline{W(k,n,m)}$ by $x_i$. In terms of Presentation~\eqref{alt_bou}, the map $\varphi$ sends $x_i$ to $b^{-i+1} a b^{i-1}$ for all $i=1, 2, \dots, n$. Writing again $m=qn+r$ for the Euclidean division of $m$ by $n$, let us observe for later use that \begin{align}\label{phidelta} \varphi(x_1 x_2\cdots x_m)&=\varphi(x_1 x_2 \cdots x_n)^q \varphi(x_1 x_2 \cdots x_r)\\&=(ab^{-1})^q(ab^{-1})^{r-1} a b^{r-1}=(ab^{-1})^m b^r=b^r.
\end{align}
In particular, as $b$ has order $n$ and $r$ and $n$ are coprime, the map $\varphi$ is surjective as both $a$ and $b$ are in the image of $\varphi$.  

Let us construct the inverse $\psi$ of $\varphi$. To this end, let $\ell\geq 1$ be the smallest positive integer such that $(b^r)^\ell=b$. Again, as $n$ and $r$ are coprime, such an integer must exist. We define $\psi$ on generators by $a\mapsto x_1$, $b\mapsto (x_1 x_2\cdots x_m)^\ell$. Let us check that these images satisfy the defining relations of Presentation~\ref{alt_bou}. We have $\psi(a)^k=x_1^k=1$ and $\psi(b)^n=((x_1 x_2\cdots x_m)^\ell)^n=(\underbrace{(x_1 x_2\cdots x_m)^n}_{=1})^\ell=1$. Now we have, taking indices modulo $n$ when necessary \begin{align*}
(\psi(b)\psi(a^{-1}))^m&=((x_1 x_2\cdots x_m)^\ell x_1^{-1})^m=((x_2 x_3\cdots x_m x_1)^{\ell-1} x_2 x_3\cdots x_m)^m.
\end{align*}  
Writing $\delta=x_1 x_2\cdots x_m$, we have $x_i \delta=\delta x_{i+r}$ for all $i=1, \dots, n$ by Lemma~\ref{delta}. Using this relation, one checks by induction on $i$ using the fact that $\ell r\equiv 1~(\mathrm{mod}~n)$ that for $i\geq 1$, we have $$(\delta^{\ell-1} x_2 x_3 \cdots x_m)^i=\delta^{i(\ell-1)} x_{m-(m-1)i+1} x_{m-(m-1)i+2}\cdots x_{m-1} x_m.$$
For $i=m$ the product $x_{m-(m-1)m+1} x_{m-(m-1)m+2}\cdots x_{m-1} x_m$ has $m(m-1)$ factors and since two consecutive factors have consecutive indices, using the defining relations $x_1 x_2 \cdots x_m=x_i x_{i+1} \cdots x_{i+m-1}$ ($i\geq 2$) we get that this product is equal to $\delta^{m-1}$. Hence $$(\psi(b)\psi(a^{-1}))^m=(\delta^{\ell-1}x_2 x_3\cdots x_m)^m=\delta^{m(\ell-1)+m-1}.$$ Now, using that $\ell r \equiv 1~(\mathrm{mod}~n)$, we get that $m(\ell-1)+m-1\equiv 0~(\mathrm{mod}~n)$, hence that $(\psi(b)\psi(a^{-1}))^m=1$ as $\delta^n=1$. This shows that $\psi$ is also a group homomorphism. 

It remains to show that $\varphi$ and $\psi$ are inverse to each other. We have $\varphi\circ\psi=\mathrm{id}$ as $a\mapsto x_1\mapsto a$ and using~\eqref{phidelta} $b\mapsto (x_1 x_2\cdots x_m)^\ell \mapsto (b^r)^\ell=b$. Conversely, the map $\psi\circ \varphi$ maps $x_i\mapsto b^{-i+1} a b^{i-1}\mapsto\delta^{(-i+1)\ell} x_1 \delta^{(i-1)\ell}$ ($i=1,\dots, n$). By Lemma~\ref{delta} we have $$x_1 \delta^{(i-1)\ell}=\delta^{(i-1)\ell} x_{1+(i-1)\ell r}$$ and as $\ell r\equiv 1 (\mathrm{mod}~ n)$, we have $x_{1+(i-1)\ell r}=x_i$, which concludes the proof. 
\end{proof}

\begin{proof}[Proof of Theorem~\ref{thm_alt}]
The isomorphism $\overline{W(k,n,m)}\cong W_{k,n,m}^+$ claimed in the first point is shown in Proposition~\ref{prop_1} above. This also establishes that the short exact sequence in the second row of the diagram of the second point is exact---this also shows that the map $G\longrightarrow W_{k,n,m}^+$ is surjective. To show that the first row is exact, it suffices to see that the presentation $$\langle s, t, u~|~s^k=t^n=u^m=1, stu=1 \rangle$$ is a presentation of $W_{k,n,m}^+$ via $s\mapsto r_1 r_2=a$, $t\mapsto r_2 r_3=b^{-1}$, $u\mapsto r_3 r_1=b a^{-1}$. But this exactly yields Presentation~\ref{alt_bou}. Now the commutativity of the diagram is clear: the commutativity of the square of the right is clear, and since the first row is exact, the kernel $\langle c \rangle$ of $W(k,n,m)\longrightarrow W_{k,n,m}^+$ has to be included in the kernel $\langle stu \rangle$ of the first short exact sequence; this can also be seen explicitly using the fact that $stu$ is central in $G$, as $$c=(x_1 x_2 \cdots x_n)^m=(st)^{nm}=(stu u^{-1})^{nm}=(stu)^{nm} (u^m)^{-n}=(stu)^{nm}.$$ 
This establishes the second point. Now by Proposition~\ref{prop_2}, the group $W_{k,n,m}^+$ has trivial center except possibly in the cases where $W_{k,n,m}$ is finite or not irreducible, that is, for $(k,n,m)=(2,2,m)$ with $m$ odd, $(2,3,4)$, $(2,3,5)$, $(3,2,3)$, $(3,2,5)$, $(4,2,3)$, $(5,2,3)$. In all cases we get finite Coxeter groups of types $A_1\times I_2(m)$ ($m$ odd), $A_3$, $B_3$ and $H_3$. The center of $A_3$ is trivial, and the center of both $H_3$ and $B_3$ has order two, generated by the longest element $w_0$, which is not in the alternating subgroup. The center of $A_1\times I_2(m)$ is also of order two (it is the $A_1$ component) since $m$ is odd, generated by a simple reflection, hence its non-trivial element is not in the alternating subgroup. We thus have seen that in all possible cases, the group $W_{k,n,m}^+$ has trivial center. This implies that the kernels of both short exact sequences are in fact the centers of the groups in the middle. This establishes point $3$, and the last point is then immediate.  
\end{proof}

\begin{rmq}
The idea of extending the group $G/\langle stu\rangle$ into a Coxeter group is present in work of Coxeter~\cite[Section 4.1]{coxeter_tams}. 
\end{rmq}

\begin{rmq}\label{center_finite}
We do not know if $c$ has finite order in $W(k,n,m)$ or not when $W(k,n,m)$ is infinite. This is equivalent to determining if $stu$ has finite order or not in $G$. In Achar and Aubert's representation (see Section~\ref{rep}), the element $\rho(stu)$ has finite order since it is a scalar matrix with eigenvalue a product of three roots of unity, but we can use the observation made in Example~\ref{ex_unfaith} to see that in general the restriction of the representation to $\langle stu \rangle$ is unfaithful: again, consider the group $W(6,2,3)$. We have seen that $\rho(c)=\rho((x_1 x_2)^3)=1$, while $c$ cannot be equal to $1$ in $W(6,2,3)$.   
\end{rmq}

We end up the section with an observation on the solvability of the word problem in $W(k,n,m)$. Recall that a group has a solvable word problem if there exists an algorithm allowing one to determine in finite time if an arbitrary word represents the identity or not. Theorem~\ref{thm_alt} seems to be close to answering this question positively since $W(k,n,m)$ is a central extension of a subgroup of a Coxeter group (which therefore has a solvable word problem as Coxeter groups have a solvable word problem~\cite[Section 2.3.3]{AB}) with a cyclic group. Nevertheless, and as observed in the previous remark, the center is not clearly identified. Given any word in the generators of $W(k,n,m)$, we can take its image in $W_{k,n,m}^+$ and say whether it represents the identity or not, hence solve the problem of determining if the word we started with represents an element which lies in the center of $W(k,n,m)$ or not. Nevertheless, if it does lie in the center, it is not clear to us how to check if it represents the identity or not. 

\begin{question}\label{quest_solv} Do the groups $W(k,n,m)$ have a solvable word problem? \end{question}

We conjecture the answer to this question to be positive.

\section{Classification of toric reflection groups}\label{main}

The aim of this section is to show the following result
\begin{theorem}[Classification of toric reflection groups]\label{thm_clas} Let $k,k',n,n',m,m'\geq 2$ with $n<m$, $n'<m'$, $n,m$ coprime, and $n',m'$ coprime. Then $$W(k,n,m)\cong_{\mathrm{ref}} W(k',n',m') \Leftrightarrow (k,n,m)=(k',n',m').$$
\end{theorem}

To this end, we will require the description of the quotient of a toric reflection group by its center from the previous section, together with the following result:

\begin{prop}[{\cite[Chap. V, Sec. 4, Exercise 2]{Bourbaki} or \cite[Proposition 1.3]{BH}}]\label{prop_BH}
Let $(W,S)$ be a Coxeter system and $H\subseteq W$ a finite subgroup of $W$. Then there exists $w\in W$ and $J\subseteq S$ such that $W_J$ is finite and $wHw^{-1}\subseteq W_J$. 
\end{prop} 

\begin{cor}\label{max_alt}
Let $(W,S)$ be a Coxeter system. Every increasing chain $$H_1 \subseteq H_2 \subseteq H_3 \subseteq \cdots $$ of finite subgroups $(H_i)_{i\geq 1}$ stabilizes, \emph{i.e.}, there is $n\geq 1$ such that $H_{n+p}=H_n$ for all $p\geq 1$. 
\end{cor}

\begin{proof}
By Proposition~\ref{prop_BH}, the cardinality of a finite subgroup of $W$ is bounded by $$N:=\max\{ |W_J|~|~J\subseteq S, ~W_J~\text{is finite} \},$$ which is well-defined since $S$ is finite. 
\end{proof}

\begin{definition}
Let $G$ be a group. A subgroup $H\subseteq G$ is \defn{maximal finite} if $H$ is finite and if for every subgroup $H'\subseteq G$, $$H\subsetneq H' \ \Rightarrow \ H' \ \text{is infinite}.$$
\end{definition}

Note that maximal finite subgroups need not exist in general. By Corollary~\ref{max_alt}, maximal finite subgroups always exist in Coxeter groups, and we even have the stronger statement that every finite subgroup of a Coxeter group is included in a maximal finite subgroup: otherwise one could build an increasing chain of finite subgroups which does not stabilize. Note that being a maximal finite subgroup is a property which is stable by conjugation.
 
\begin{prop}\label{cox_alt}
Let $(W,S)$ be a Coxeter system and $$M_W:=\{ I\subseteq S~|~\forall J\subseteq S, I\subsetneq J \Rightarrow W_J~\text{is infinite}\}.$$ Then $\{ W_I~|~I\in M_W\}$ is a set of representatives of the conjugacy classes of maximal finite subgroups of $W$.  
\end{prop}

\begin{proof}
By Proposition~\ref{prop_BH}, if $H\subseteq W$ is finite, then $w H w^{-1} \subseteq W_J$ for some $w\in W$ and some $J\subseteq S$ such that $W_J$ is finite. To conclude the proof, it therefore suffices to show that for $I\in M_W$, the subgroup $W_I$ is maximal finite, and that for $I\neq J$ where $I,J\in M_W$, the subgroups $W_I$ and $W_J$ are not conjugate to each other. This last property is a consequence of~\cite[Corollary 3.1.7]{Krammer}, as $I$ is a maximal subset of $S$ such that $W_I$ is finite. Now let us show that $W_I$, $I\in M_W$, is maximal finite. Let $W_I\subseteq H$ such that $H$ is a finite subgroup of $W$. By Proposition~\ref{prop_BH}, there is $J\subseteq S$ and $w\in W$ such that $W_J$ is finite and $W_I\subseteq H \subseteq w^{-1} W_J w$. Setting $w_1$ for the unique element of minimal length in $W_J w W_I$, we have $I\cap w_1^{-1} J w_1=I$ by~\cite[Lemma 2.25]{AB}, hence $I\subseteq w_1^{-1} J w_1$. Let $I'\subseteq J$ such that $I=w_1 I' w_1^{-1}$. Then since $W_I$ and $W_{I'}$ cannot be conjugate to each other if $I\neq I'$ because of the maximality of $I$ (again by~\cite[Corollary 3.1.7]{Krammer}), we have $I'=I$, hence $I\subseteq J$, which forces $I=J$ since $W_J$ is finite.  Since $W_I$ and $w^{-1} W_J w=w^{-1} W_I w$ have the same cardinality and $W_I\subseteq H\subseteq w^{-1} W_J w$, we get $H=W_I$.      
\end{proof}

Note that, since every increasing chain of finite subgroups of a Coxeter group $W$ stabilizes, the same is true for any subgroup of $W$, in particular it holds true for $W^+$. In particular, there are maximal finite subgroups in $W^+$, and every finite subgroup of $W^+$ is included in a maximal finite one. 

\begin{prop}\label{prop_max_fg}
Let $k,n,m\geq 2$ be such that $W_{k,n,m}$ is infinite. There are three conjugacy classes of maximal finite subgroups of $W_{k,n,m}^+$. The finite groups in these three classes are isomorphic to $C_k$, $C_n$, $C_m$, where $C_i$ denotes the cyclic group of order $i$. 
\end{prop}

\begin{proof}
Note that under these assumptions, the Coxeter group $W=W_{k,n,m}$ is irreducible. Moreover, in this case $M_W$ consists of the three subsets of $S$ of cardinality two as $|S|=3$ and $W$ is infinite. Let $H\subseteq W_{k,n,m}^+$ be a finite subgroup of $W_{k,n,m}^+$. We claim that $H$ is a conjugate (in $W_{k,n,m}^+$) of a subgroup of the alternating subgroup $W_J^+$ of a standard parabolic subgroup $W_J$ of rank two of $W_{k,n,m}$. If $H=1$ the claim is trivially true, hence assume that $H\neq 1$. By Proposition~\ref{prop_BH}, there is $w\in W_{k,n,m}$ such that $wHw^{-1}\subseteq W_J$ for some $J\subseteq S$ such that $W_J$ is finite. If $w\in W_{k,n,m}^+$, then $H$ and $wHw^{-1}$ are conjugate in $W_{k,n,m}^+$, hence the claim holds true. If $w\notin W_{k,n,m}^+$, then since $H\neq 1$ we have $J\neq \emptyset$, so let $s\in J$. Then $sw\in W_{k,n,m}^+$ and $sw H w^{-1}s\subseteq s W_J s= W_J$, hence $H$ is again conjugate in $W_{k,n,m}^+$ to a subgroup of $W_J$. In all cases, since $1\neq H\subseteq W_{k,n,m}^+$ and $W_J$ is finite, we must have $|J|=2$ since for $|J|=1$ we have $W_J\cap W_{k,n,m}^+=\{1\}$. Hence $H$ is conjugate to a subgroup of the alternating subgroup of one of the three standard parabolic subgroups of rank $2$ of $W_{k,n,m}$. Note that these parabolic subgroups are dihedral groups of order $2k, 2n$, or $2m$. It follows that there are at most three conjugacy classes of maximal finite subgroups of $W_{k,n,m}^+$, since if $H$ is maximal finite, then it has to be a conjugate of one of these alternating subgroups, namely $C_k$, $C_n$ or $C_m$.  

To conclude the proof, it therefore remains to show that the above three alternating subgroups of standard parabolic subgroups of rank two of $W_{k,n,m}$ are maximal finite subgroups of $W_{k,n,m}^+$ which are not conjugate to each other---in fact we shall show that they are not conjugate in $W_{k,n,m}$, which is stronger. Maximality follows again from Proposition~\ref{prop_BH}: let $J\subseteq S$ with $|J|=2$. Then if $W_J^{+}\subseteq W_J$ is not maximal finite inside $W_{k,n,m}^+$, then there is a finite subgroup $H$ of $W_{k,n,m}^+$ such that $W_J^+\subsetneq H$. By Proposition~\ref{prop_BH}, the subgroup $H$ is included in a finite parabolic subgroup $x W_{J'} x^{-1}$ and up to enlarging $J'$ we can assume that $|J'|=2$ (alternatively we can use the fact that every finite subgroup of $W_{k,n,m}$ is included in a maximal finite one together with Proposition~\ref{cox_alt}). Writing $J=\{t_1, t_2\}$, we have $t_1 t_2\in W_J^{+}\subseteq H\subseteq x W_{J'} x^{-1}$. Now the intersection of two parabolic subgroups is again a parabolic subgroup (see for instance~\cite[Lemma 2.25]{AB}), hence we have that $W_J\cap x W_{J'} x^{-1}$ is finite parabolic of rank at most two. But since it contains $t_1 t_2$, it must have rank $2$. As rank two parabolic subgroups are maximal finite by Proposition~\ref{cox_alt}, we have that both $W_J$, $x W_{J'} x^{-1}$ and $W_J\cap x W_{J'} x^{-1}$ are maximal finite, which forces $W_J=x W_{J'} x^{-1}$. Hence $W_J^{+}\subseteq H \subseteq W_J$, which forces $H=W_J^+$ as $H\subseteq W_{k,n,m}^+$. Hence $W_J^+$ is maximal finite. Now for $J_1\neq J_2$ with $J_i\subseteq S$ and $|J_1|=2=|J_2|$, then $W_{J_1}^+$ and $W_{J_2}^+$ cannot be conjugate to each other in $W$: if this was the case, there would be $x\in W$ such that $W_{J_1}^+ \subseteq x W_{J_2} x^{-1}$, and arguing as above (with $J=J_1$ and $J'=J_2$) we would get that $W_{J_1}$ and $W_{J_2}$ are conjugate to each other, contradicting Proposition~\ref{cox_alt}. 
\end{proof}

\begin{proof}[Proof of Theorem~\ref{thm_clas}]
Assume $W(k,n,m)\cong_{\mathrm{ref}} W(k',n',m')$. We first claim that $k-1$ is the number of conjugacy classes of reflections in $W(k,n,m)$. Indeed, all generators $x_i$'s are conjugate to each other, and by Theorem~\ref{part_I_main}, we have that $x_1$ is the generator $s$ of the isomorphic $J$-group $H=J\begin{pmatrix} k & n & m\\ ~& n& m\end{pmatrix}\trianglelefteq J\begin{pmatrix} k & n & m\\ ~& ~& ~\end{pmatrix}=G$. We have that $s$ has order $k$ in $G$. To conclude, since reflections in $W(k,n,m)$ are defined to be the conjugates of the non-trivial powers of the $x_i$'s, it therefore suffices to show that no two reflections in $\{s, s^2, \dots, s^{k-1}\}$ are conjugate to each other. This holds true in $G$ by Lemma~\ref{lem:conj_class}, hence it holds true a fortiori in $H$. Since reflection isomorphisms map reflections to reflections, we deduce that $k=k'$.

Now, by Theorem~\ref{thm_alt}, the group $\overline{W}(k,n,m)$ is the quotient of $W(k,n,m)$ by its center. We therefore have $$\overline{W}(k,n,m)\cong \overline{W}(k,n',m').$$ In particular, we have $W_{k,n,m}^+\cong W_{k,n',m'}^+$. We will conclude by showing that this forces $n=n'$ and $m=m'$. We first assume that $W_{k,n,m}$ is infinite, which forces $W_{k,n',m'}$ to be also infinite since alternating subgroups of Coxeter groups are subgroups of index two. By Proposition~\ref{prop_max_fg} above, there are three conjugacy classes of maximal finite subgroups of $W_{k,n,m}^+$ (resp. $W_{k,n',m'}^+$), and the isomorphism type of these finite subgroups is given by $C_k, C_n$, and $C_m$ (resp. $C_k, C_{n'}$, and $C_{m'}$). The multiset of isomorphism type of subgroups in conjugacy classes of maximal finite subgroups is obviously invariant under isomorphism, which implies that the multisets $\{k,n,m\}$ and $\{k,n',m'\}$ are equal. Since $n<m$ and $n'<m'$, this forces $n=n'$ and $m=m'$.

We now assume that $W_{k,n,m}$ (and hence $W_{k,n',m'}$) is finite. We then deduce from Table~\ref{table_gzg} that no two groups $W_{k,n,m}^+$ and $W_{k,n',m'}^+$ are isomorphic when $(n',m')\neq (n,m)$. Hence $n=n'$ and $m=m'$.  
\end{proof}

\begin{rmq}
Mimicking the definition given for finite complex reflection groups in~\cite{BMR}, one can define a Hecke algebra $\mathcal{H}_{k,n,m}$ of a toric reflection group $W(k,n,m)$ over a suitable base ring, deforming the group algebra $\mathbb{Z}[W(k,n,m)]$. It is tempting to conjecture that this algebra is a free module over this ring, with a basis deforming the basis of $\mathbb{Z}[W(k,n,m)]$ given by the elements of the group $W(k,n,m)$. However, already in the finite case, the proof of the BMR freeness conjecture for those finite complex reflection groups which are toric reflection groups is case-by-case.  
\end{rmq}

\begin{rmq}\label{interpret_param}
The parameter $k-1$ is the number of conjugacy classes of reflections in $W(k,n,m)$. It would be desirable to have an interpretation of $n$ and $m$ in terms of the "reflection group" structure of $W(k,n,m)$. When $W(k,n,m)$ is finite, the parameter $n$ is the reflection rank of $W(k,n,m)$, that is, the minimal number of reflections which are needed to generate $W(k,n,m)$. We can conjecture that this still holds true for arbitrary toric reflection groups. This would be a first step towards a proof of Theorem~\ref{thm_clas} avoiding a recourse to the theory of Coxeter groups---which has other advantages, as for instance rank three Coxeter groups have nice geometric realizations. It would also reprove that the \textit{meridional rank} of the torus knot $T_{n,m}$ ($n<m$) is equal to $n$ (this fact is proven in~\cite{merid_torus} in the context of the meridional rank conjecture): it is indeed at most $n$ since $G(n,m)$ has its classical presentation having exactly $n$ meridians as generators, and if it was smaller, then because of the surjection $G(n,m)\twoheadrightarrow W(k,n,m)$ which maps meridians to reflections, the groups $W(k,n,m)$ could then be generated by less than $n$ reflections. As other $J$-groups might be reflection quotients of other link groups, this could be of interest, even if in the specific case of torus knot groups everything seems to be already known.    
\end{rmq}

\textbf{Acknowledgments.} I thank Anne-Marie Aubert for an email exchange on $J$-groups. I thank Jean Michel for pointing out reference~\cite{MM} to my attention. I thank Gunter Malle for his careful reading of the manuscript and suggestions. The unfaithfulness of Achar and Aubert's representation, discussed in Section~\ref{rep}, was easy to establish after noticing that Coxeter's representation from~\cite{coxeter_factor} is unfaithful: this last property was noticed during discussions with Vincent Beck on truncated braid groups. I also thank him.

\end{document}